\documentclass{article}
\usepackage[a4paper]{geometry}
\usepackage{amssymb, amsthm, amsmath}
\usepackage[utf8]{inputenc}
\usepackage[T1]{fontenc}
\usepackage{lmodern}
\usepackage{csquotes}
\usepackage[english]{babel}
\usepackage{mathtools}
\usepackage{tikz}
\usetikzlibrary{fadings}
\usepackage{pgfplots}\pgfplotsset{compat=1.18}
\usepackage{enumerate}
\usepackage{mathrsfs}
\usepackage{accents}
\usepackage[capitalise]{cleveref}
\crefname{equation}{}{}
\allowdisplaybreaks

\theoremstyle{plain}
\newtheorem{theorem}{Theorem}
\newtheorem{lemma}[theorem]{Lemma}
\newtheorem{corollary}[theorem]{Corollary}
\theoremstyle{definition}
\newtheorem{definition}[theorem]{Definition}
\newtheorem{remark}[theorem]{Remark}

\newtheorem{example}[theorem]{Example}

\newcommand{\R}{\mathbb{R}}

\newcommand{\N}{\mathbb{N}}
\newcommand{\norm}[1]{\left\Vert#1\right\Vert}
\newcommand{\normsmall}[1]{\Vert#1\Vert}
\newcommand{\Xnorm}[1]{\left\Vert#1\right\Vert_X}
\newcommand{\xnorm}[1]{\Vert#1\Vert_X}
\newcommand{\setnorm}[1]{|#1|}

\newcommand{\bracket}[1]{\left[#1\right]}

\newcommand{\set}[1]{\left\lbrace#1\right\rbrace}
\newcommand{\setsmall}[1]{\lbrace#1\rbrace}

\newcommand{\klammern}[1]{\left(#1\right)}
\newcommand{\floor}[1]{\lfloor#1\rfloor}
\newcommand{\Floor}[1]{\left\lfloor#1\right\rfloor}
\newcommand{\ceiling}[1]{\lceil#1\rceil}

\newcommand{\diff}{\mathop{}\!\mathrm{d}}
\newcommand\numberthis{\addtocounter{equation}{1}\tag{\theequation}}
\newcommand{\ds}{\displaystyle}
\newcommand{\define}[1]{\emph{#1}}
\newcommand{\variation}[1]{\Var\left(#1\right)}
\newcommand{\discrete}{\theta}
\newcommand{\Newline}{\displaybreak[0]\\}
\newcommand{\CP}[1]{$\textnormal{(CP}\colon#1\textnormal{)}$}
\newcommand{\C}{{C([0, T]; X)}}

\newcommand{\Anorm}[1]{\left|#1\right|_A}

\newcommand{\truestatements}[1]{\begin{enumerate}[{\normalfont (a)}]#1\end{enumerate}}
\DeclareMathOperator{\range}{range}
\DeclareMathOperator{\dom}{dom}
\DeclareMathOperator{\supp}{supp}
\DeclareMathOperator{\Var}{Var}
\DeclareMathOperator{\essVar}{essVar}
\DeclareMathOperator{\generalizeddom}{\widetilde{dom}}

\title{Differences of solutions of implicit Euler schemes with accretive operators on Banach spaces}
\author{Johann Beurich}
\date{}

\begin{document}

\maketitle

\begin{abstract}
    We give an upper bound for the difference of two solutions of Euler schemes approximating the Cauchy problem
    \[\begin{cases}
	\dot{u}(t) + Au(t) \ni f(t) \qquad (t \in [0, T]), \\
	u(0) = u^0,
    \end{cases}\]
    where $A \subseteq X \times X$ is a quasi-accretive operator on a Banach space $X$, $T > 0$, $f \in L^1(0, T; X)$ and $u^0 \in X$.
    This upper bound generalizes a result from Kobayashi, who established an upper bound for the problem with $f = 0$ in \cite{Kobayashi75}. We show, that the upper bound can be used to establish existence and uniqueness of Euler solutions as limits of solutions of Euler schemes as well as regularity of Euler solutions.
\end{abstract}

\section{Introduction}

Let $x$ be a Banach space with norm $\Xnorm{\cdot}$. We consider the Cauchy problem
\[ \tag{$\mathrm{CP} \colon f, u^0$} \label{Cauchy-Problem}
\begin{cases}
	\dot{u}(t) + Au(t) \ni f(t) \qquad (t \in [0, T]), \\
	u(0) = u^0,
\end{cases}
\]
where $A \subseteq X \times X$, $T > 0$, $f \in L^1(0, T; X)$ and $u^0 \in X$. We assume that the operator $A$ is accretive of type $\omega \in \R$. We call uniform limits of solutions of implicit Euler schemes corresponding to the Cauchy problem, {Euler solutions}.
The fundamental questions of existence and uniqueness of Euler solutions were first discussed in the late 1960s and early 1970s. 
Crandall and Liggett \cite{CrandallLiggett1971} showed in 1971 existence and uniqueness of Euler solutions, if $f = 0$ and $A$ 
satisfies some range condition.

For $f \neq 0$ and $A$ being $m$-accretive of type $\omega$, Crandall and Evans were able to show existence and uniqueness of Euler solutions in 1975 \cite{CrandallEvans1975} by comparing norms of differences of solutions of Euler schemes to exact solutions of boundary value problems involving the differential operator $\partial /\partial s+\partial /\partial \tau$.
In the same year Kobayashi \cite{Kobayashi75} found an elegant way of estimating the difference of two solutions of Euler schemes, which gave a direct way of showing existence and uniqueness of Euler solutions for $f = 0$. He also established the rate of convergence $O({|\pi|}^{1/2})$, where $|\pi|$ is the mesh size of the time partition $\pi$ used in the Euler schemes. Kobayashi was aware, that his results overlapped with the results from Crandall and Evans in \cite{CrandallEvans1975}. He wrote \cite[Remark 2.2]{Kobayashi75}:
\blockquote{
	After the preparation of this manuscript, Prof. M. Crandall informed me that Crandall and Evans \cite{CrandallEvans1975} proved Theorem 2.1 by an entirely different method, which is interesting in itself. They treat a more general evolution equation
	[...]
	where $f\in L^{1}(0, T;X)$ is given. Our method is also applicable to this case.
}
However, Kobayashi's result stated in \cite{Kobayashi75} is not sufficient to show convergence of solutions of Euler schemes for $f \neq 0$. Crandall and Evans also mention this. They wrote \cite{CrandallEvans1975}:
\blockquote{
	Kobayashi's note came to the attention of the authors after most of the research in the current paper was complete. There is some minor intersection of our development with that of \cite{Kobayashi75}. The case $f \neq 0$ seems genuinely more complex than the case $f = 0$, and our main point is not only Theorem 1.2 but its proof, which is of independent interest.
}
Our main result \Cref{main-result} is a generalization of Kobayashi's result.
We obtain an explicit estimate for the norm of differences fo solutions of implicit Euler schemes even in the case $f \neq 0$.
From there one deduces easily existence of solutions, a priori estimates (obtained earlier by Nochetto and Savar{\'e} \cite{NoSa2006}), stability of solutions and regularity.

\section{Euler solutions and accretive operators}

We use the following notation.
We call any set $A \subseteq X \times X$ an \define{operator} on $X$ and we define
	\begin{align*}
		\dom A &\coloneqq \set{u \in X \colon (u, v) \in A \text{ for some } v \in X } , \\
		\range A &\coloneqq \set{v \in X \colon(u, v) \in A \text{ for some } u \in X} , \\
		Au & \coloneqq \set{v \in X \colon (u, v) \in A} \quad \text{for } u \in X , \\
		\lambda A &\coloneqq \set{(u, \lambda v) \colon (u, v) \in A} \quad \text{for } \lambda \in \R, \\
		A + B &\coloneqq \set{(u, v+w) \colon (u, v) \in A \text{ and } (u, w) \in B} \quad \text{for } B \subseteq X \times X , \\
		A + x &\coloneqq \set{(u, v+x) \colon (u, v) \in A} \quad \text{for } x \in X , \\
		A^{-1} &\coloneqq \set{(v, u) \colon (u, v) \in A}  \quad \text{and} \\
		I &\coloneqq \set{(u, u) \colon u \in X} .
	\end{align*}
Let $a, b \in \R$, $a < b$. Any finite subset \[ \set{t_0, t_1, \ldots, t_N} \subseteq [a, b] \] containing $a$ and $b$ is called a \define{partition} of the compact interval $[a, b]$. We usually think, without loss of generality, that the elements $t_0, t_1, \ldots, t_N \in \pi$ are ordered, and we write
\[ \pi \colon a = t_0 < t_1 < \ldots < t_N = b . \]
For every partition $\pi \colon a = t_0 < t_1 < \ldots < t_N = b$ we define the \define{time steps}
\[ h_i \coloneqq t_{i+1} - t_i \]
for every $i \in \set{0, \ldots, N-1}$ and the \define{mesh size} of the partition
\[ |\pi| \coloneqq \sup \set{h_i : i \in \set{0, \ldots, N-1}} . \]
We further define the \define{floor function} $\floor{\cdot}_{\pi} \colon \R \to \R$ by
\[ \floor{s}_{\pi} \coloneqq \begin{cases}
    t_i &\text{if $s \in [t_i, t_{i+1})$, $i \in \set{0, \ldots, N-1}$,} \\
    s &\text{if $s \not\in [a, b)$},
\end{cases} \]
and the \define{ceiling function} $\ceiling{\cdot}_{\pi} \colon \R \to \R$ by
\[ \ceiling{s}_{\pi} \coloneqq \begin{cases}
    t_{i+1} &\text{if $s \in [t_i, t_{i+1})$, $i \in \set{0, \ldots, N-1}$,} \\
    s &\text{if $s \not\in [a, b)$}.
\end{cases} \]
A function $f \colon [a, b] \to X$ is \define{adapted to the partition $\pi$}, if
\[ \text{$f$ is constant on $[t_i, t_{i+1})$ for every $i \in \set{0, \ldots, N-1}$.} \]
A \define{discretization} $\discrete$ is any triple of the form $(\pi_{\discrete} , f_\discrete , u_\discrete^0)$, where $\pi_\discrete$ is a partition of the interval $[0, T]$, $f_\discrete \colon [0, T] \to X$ is a step function adapted to the partition $\pi_\discrete$ and $u_\discrete^0 \in X$.

For a given operator $A \subseteq X \times X$ and a discretization $\discrete = (\pi_{\discrete} , f_\discrete , u_\discrete^0 )$ we consider the implicit Euler scheme
\begin{equation}
    \begin{cases} \ds
        \frac{u_\discrete(t_{i+1})-u_\discrete(t_i)}{h_i} + A u_\discrete(t_{i+1}) \ni f_\discrete(t_{i}) \text{ for $i \in \set{0, \ldots , N-1}$} , \\
        u_\discrete (0) = u_\discrete^0 .
    \end{cases}
    \label{Euler-scheme} \tag{$E_\discrete$}
\end{equation}

We say that $u_\discrete \in C([0,T]; X)$ is a \define{solution of the implicit Euler scheme \Cref{Euler-scheme}}, if \Cref{Euler-scheme} holds and $u_\discrete$ is affine on the intervals $[t_i, t_{i+1}]$, that is,
\[ u_\discrete (t) = \frac{t_{i+1}-t}{t_{i+1}-t_i} u_\discrete (t_{i}) + \frac{t-t_i}{t_{i+1}-t_i} u_\discrete (t_{i+1}) \]
for all $t \in [t_i, t_{i+1}]$ and all $i \in \set{0, \ldots, N-1}$.

Let $f \in L^1(0, T; X)$ and $u^0 \in X$. We call a function $u \in C([0, T];X)$ an \define{Euler solution} of \Cref{Cauchy-Problem}
	if there exists a sequence of discretizations $(\discrete_n) = ((\pi_{\discrete_n} , f_{\discrete_n} , u_{\discrete_n}^0 ))$ and a sequence $(u_{\discrete_n})_n$ in $C([0, T]; X)$, such that $u_{\discrete_n}$ is a solution of the implicit Euler scheme $(E_{\discrete_n})$ for all $n \in \N$,
	\begin{align*}
		&\lim_{n \to \infty} |\pi_{\discrete_n}| = 0, \\
		&\lim_{n \to \infty} \norm{f_{\discrete_n} - f}_{L^1(0, T; X)} = 0 \text{ and} \\
		&\lim_{n \to \infty} \norm{u_{\discrete_n} - u}_{C([0, T]; X)} = 0 ,
	\end{align*}
In the literature Euler solutions of \Cref{Cauchy-Problem} are often called mild solutions of \Cref{Cauchy-Problem}.

The \define{bracket} for $u, v \in X$ is defined as
\[ \bracket{u, v} \coloneqq \inf_{\lambda > 0} \frac{ \Xnorm{u + \lambda v} - \Xnorm{u}}{\lambda} . \]
For properties of the bracket we refer the reader to \cite{Bar10}.

An operator $A \subseteq X \times X$ is \define{accretive of type $\omega \in \R$} if
	\[ \bracket{u - \hat u, v - \hat v} + \omega \Xnorm{u - \hat u} \ge 0 \numberthis \label{accretive-definition} \]
	for every $(u,v), (\hat u, \hat v) \in A$.
 The operator $A$ is \define{accretive} if $A$ is accretive of type 0 and $A$ is \define{quasi-accretive} if $A$ is accretive of type $\omega$ for some $\omega \in \R$.
	
	If $A$ is accretive of type $\omega \in \R$ and
	\[ \range (I + \lambda A) = X \]
	for all $\lambda > 0$ with $\lambda \omega < 1$, then $A$ is \define{$m$-accretive of type $\omega$}. The operator $A$ is \define{$m$-accretive} if $A$ is $m$-accretive of type 0 and $A$ is \define{quasi-$m$-accretive} if $A$ is $m$-accretive of type $\omega$ for some $\omega \in \R$.
	
	A quasi-accretive operator $A$ is said to satisfy the \define{range condition} if there exists $\lambda_0 > 0$ such that
	\[ \overline{\dom A} \subseteq \range(I + \lambda A) \]
	for every $\lambda \in (0, \lambda_0)$. Clearly, every quasi-$m$-accretive operator satisfies the range condition.

\section{Differences of solutions of implicit Euler schemes}

In this section we take two discretizations $\discrete = (\pi_\discrete, f_\discrete, u_\discrete^0)$ and $\hat \discrete = (\pi_{\hat\discrete}, f_{\hat\discrete}, u_{\hat\discrete}^0)$. Let $u_\discrete \in \C$ be a solution of the implicit Euler scheme \Cref{Euler-scheme} and $u_{\hat\discrete} \in \C$ be a solution of $(E_{\hat\discrete})$.
For $i \in \set{0,\ldots,N}$ and $j \in \setsmall{0,\ldots,\hat{N}}$ we define
\[ a_{i,j} \coloneqq \xnorm{u_\discrete (t_i) - u_{\hat{\discrete}} (\hat{t}_j)} . \]
Our goal is to find an estimate for $a_{i,j}$, which only depends on the discretizations $\discrete$ and $\hat\discrete$.

To see that $a_{i, j}$ gets small, if $|t_i - \hat t_j|$, $|\pi_\discrete|$, $|\pi_{\hat\discrete}|$, $\xnorm{u_\discrete^0 - u_{\hat\discrete}^0}$ and $\normsmall{f_\discrete - f_{\hat\discrete}}_{L^1(0, T; X)}$ are small, we want to obtain an upper bound, which works for all partitions with small mesh sizes $|\pi_\discrete|$ and $|\pi_{\hat\discrete}|$.

\subsection{An implicit upper bound} \label{section-2-1-implicit-uppter-bound}

\begin{lemma} \label{induktionsschritt}
For $i \in \set{0,\ldots,N-1}$ and $j \in \setsmall{0,\ldots,\hat{N}-1}$ we have
\begin{align*}
    (1 - (h_i \land \hat{h}_j)\omega) a_{i+1,j+1} \le {}&\bigg(1- \frac{\hat{h}_j}{h_i}\bigg)^+ a_{i+1,j} + \bigg(1- \frac{h_i}{\hat{h}_j}\bigg)^+ a_{i,j+1} + \frac{h_i \land \hat{h}_j}{h_i \lor \hat{h}_j} a_{i,j} \\
    & + (h_i \land \hat{h}_j) \bracket{u_\discrete (t_{i+1}) - u_{\hat{\discrete}} (\hat{t}_{j+1}) , f_\discrete (t_i) - f_{\hat{\discrete}} (\hat{t}_j)} . 
 \numberthis \label{omegaIstNochLinks}
\end{align*}
\end{lemma}

\begin{proof}
Let us first assume $\hat{h}_j \le h_i$. Using the Euler schemes \Cref{Euler-scheme} and ($E_{\hat\discrete}$), that $A$ is accretive of type $\omega$ and properties of the bracket (see \cite[Proposition 3.7]{Bar10}), we get
\begin{align*}
    0 \le{}& \hat{h}_j \Bigg[ u_\discrete (t_{i+1}) - u_{\hat{\discrete}} (\hat{t}_{j+1}) , f_{\discrete} (t_i) - \frac{u_\discrete (t_{i+1}) - u_\discrete (t_i)}{h_i} - f_{\hat{\discrete}} (\hat{t}_j) + \frac{u_{\hat{\discrete}} (\hat{t}_{j+1}) - u_{\hat{\discrete}} (\hat{t}_j)}{\hat{h}_j} \Bigg] \\*
    & + \hat{h}_j \omega \Xnorm{u_\discrete (t_{i+1}) - u_{\hat{\discrete}} (\hat{t}_{j+1})} \\
    {}={}& \Bigg[ u_\discrete (t_{i+1}) - u_{\hat{\discrete}} (\hat{t}_{j+1}) , \hat{h}_j ( f_{\discrete} (t_i) - f_{\hat{\discrete}} (\hat{t}_j) ) - (u_\discrete (t_{i+1}) - u_{\hat{\discrete}} (\hat{t}_{j+1})) \\*
    & + \bigg(1- \frac{\hat{h}_j}{h_i}\bigg) (u_\discrete (t_{i+1}) - u_{\hat{\discrete}} (\hat{t}_j)) + \frac{\hat{h}_j}{h_i} (u_\discrete (t_i) - u_{\hat{\discrete}} (\hat{t}_j)) \Bigg] + \hat{h}_j \omega a_{i+1,j+1} \\
    {}\le{} & \hat{h}_j \bracket{u_\discrete (t_{i+1}) - u_{\hat{\discrete}} (\hat{t}_{j+1}) , f_\discrete (t_i) - f_{\hat{\discrete}} (\hat{t}_j)} - a_{i+1,j+1} \\*
    &+ \bigg(1- \frac{\hat{h}_j}{h_i}\bigg) a_{i+1,j} + \frac{\hat{h}_j}{h_i} a_{i,j} + \hat{h}_j \omega a_{i+1,j+1} .
\end{align*}
Rearranging the terms gives us the desired inequality.
The case $\hat{h}_j > h_i$ can be treated analogously.
\end{proof}

\begin{remark}
	For $i \in \set{0,\ldots,N-1}$ and $j \in \setsmall{0,\ldots,\hat{N}-1}$ we also have
	\begin{align*}
		(1 - (h_i \land \hat{h}_j)\omega) a_{i+1,j+1} \le {}& \frac{h_i}{h_i + \hat h_j} a_{i+1,j} + \frac{\hat h_j}{h_i + \hat h_j} a_{i,j+1} \\
		& + \frac{h_i \hat h_j}{h_i + \hat h_j} \bracket{u_\discrete (t_{i+1}) - u_{\hat{\discrete}} (\hat{t}_{j+1}) , f_\discrete (t_i) - f_{\hat{\discrete}} (\hat{t}_j)} .
	\end{align*}
	Many authors work with this version of an iterative estimate for $a_{i+1, j+1}$.
\end{remark}

\begin{definition}
	We define the continuous and strictly increasing function $\varphi \colon (-\infty, 1) \to \R$ by
\[ \varphi(x) \coloneqq \begin{cases} \ds \frac{-\log(1-x)}{x} &\text{if } x \neq 0, \\ 1 &\text{if } x = 0. \end{cases} \]
\begin{figure}[htbp]
\centering
\begin{tikzpicture}[scale=1.5]
\draw[->] (-1.5, 0) -- (1.5, 0) node[right] {$x$};
\draw[->] (0, -0.3) -- (0, 2.5);
\draw (1,0.05) -- (1,-0.05) node[below] {1};
\draw (0.05,1) -- (-0.05,1) node[left] {1};
\draw[domain=-1.5:0.895, smooth, variable=\x, very thick] plot ({\x}, {-ln(1-\x)/\x});
\draw (0.83, 2.67) node {$\varphi(x)$};
\end{tikzpicture}
\caption{Graph of $\varphi$.}
\end{figure}
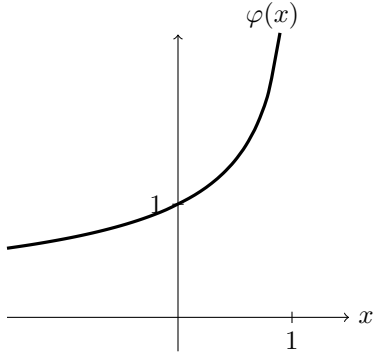
\end{definition}

\begin{lemma} \label{bessereOmegaAbschaetzung}
If $(|\pi_{\discrete}| \land |\pi_{\hat\discrete}|) \omega < 1$, then
\begin{align*}
    a_{i+1,j+1} \le {} & \mathopen{\exp} \left( \varphi((|\pi_\discrete|\land|\pi_{\hat{\discrete}}|)\omega) (h_i \land \hat{h}_j) \omega \right) \Bigg( \bigg(1- \frac{\hat{h}_j}{h_i}\bigg)^+ a_{i+1,j} + \bigg(1- \frac{h_i}{\hat{h}_j}\bigg)^+ a_{i,j+1} \\ & + \frac{h_i \land \hat{h}_j}{h_i \lor \hat{h}_j} a_{i,j} + (h_i \land \hat{h}_j) \bracket{u_\discrete (t_{i+1}) - u_{\hat{\discrete}} (\hat{t}_{j+1}) , f_\discrete (t_i) - f_{\hat{\discrete}} (\hat{t}_j)} \Bigg)
\end{align*}
for $i \in \set{0,\ldots,N-1}$ and $j \in \setsmall{0,\ldots,\hat{N}-1}$.
\end{lemma}

\begin{proof}
By assumption we have $1 - (h_i \land \hat{h}_j) \omega > 0$, so we can use \Cref{induktionsschritt} and divide both sides of \Cref{omegaIstNochLinks} by $1 - (h_i \land \hat{h}_j) \omega $. Since $\exp$ and $\varphi$ are increasing, we get
\begin{align*}
	\frac{1}{1 - (h_i \land \hat{h}_j) \omega} &= \mathopen{\exp} \left( - \log (1 - (h_i \land \hat{h}_j) \omega) \right) \\
	&= \mathopen{\exp} \left( \varphi((h_i \land \hat{h}_j)\omega) (h_i \land \hat{h}_j) \omega \right) \\
	&\le \mathopen{\exp} \left( \varphi((|\pi_\discrete|\land|\pi_{\hat{\discrete}}|)\omega) (h_i \land \hat{h}_j) \omega \right).
\end{align*}
Note that for $\omega < 0$ we get $\varphi((h_i \land \hat{h}_j)\omega) \ge \varphi((|\pi_\discrete|\land|\pi_{\hat{\discrete}}|)\omega)$, which still works for the estimate above. This completes the proof.
\end{proof}

\begin{lemma}\label{induktionsanfang-v2}
If $|\pi_\discrete| \omega < 1$, then for every $(u, v) \in A$ and for every $i \in \set{0, \ldots , N}$ we have
\begin{align*}
	\Xnorm{u_\discrete(t_i) - u} \le {} & \mathopen{\exp} \left({\varphi(|\pi_\discrete| \omega) t_i \omega} \right) \xnorm{u_\discrete^0 - u} \\*
 & + \int_0^{t_i} \mathopen{\exp} \left({\varphi(|\pi_\discrete| \omega) (t_i - \floor{\tau}_{\pi_\discrete} ) \omega} \right) \bracket{ u_\discrete ( \ceiling{\tau}_{\pi_\discrete} ) - u , f_\discrete (\tau) - v } \diff\tau .
 \numberthis \label{induktionsanfang_i}
\end{align*}
If $|\pi_{\hat\discrete}| \omega < 1$, then, as above, for every $(u, v) \in A$ and every $j \in \setsmall{0, \ldots, \hat N}$ we have
\begin{align*}
	\xnorm{u_{\hat\discrete}(\hat t_j) - u} \le {} & \mathopen{\exp} \left({\varphi(|\pi_{\hat\discrete}| \omega) \hat t_j \omega} \right) \xnorm{u_{\hat\discrete}^0 - u} \\*
 & + \int_0^{\hat t_j} \mathopen{\exp} \left({\varphi(|\pi_{\hat\discrete}| \omega) (\hat t_j - \floor{\tau}_{\pi_{\hat\discrete}} ) \omega} \right) \bracket{ u_{\hat\discrete} ( \ceiling{\tau}_{\pi_{\hat\discrete}} ) - u , f_{\hat\discrete} (\tau) - v } \diff\tau .
 \numberthis \label{induktionsanfang_j}
\end{align*}
\end{lemma}

\begin{proof}
We prove \Cref{induktionsanfang_i} by induction over $i$. For $i = 0$ both sides are equal to \( \Xnorm{u_\discrete^0 - u} \). If the statement is true for some $i \in \set{0, \ldots, N-1}$, then note that
\[ \frac{u-u}{h_i} + Au \ni v , \]
so we can use \Cref{bessereOmegaAbschaetzung} with the discretization $\hat \discrete = (\pi_\discrete, v, u)$ and constant solution $u_{\hat \discrete} = u$ of the implicit Euler scheme ($E_{\hat\discrete}$) to get
\begin{align*}
    \Xnorm{u_\discrete (t_{i+1}) - u} \le {} & \exp \klammern{\varphi (|\pi_\discrete| \omega) h_i \omega} \big( \Xnorm{u_\discrete (t_i) - u} + h_i \bracket{u_\discrete(t_{i+1})-u, f_\discrete(t_i) - v} \big) \\
    {}={}& \exp \klammern{\varphi (|\pi_\discrete| \omega) h_i \omega} \Xnorm{u_\discrete (t_i) - u} \\*
    & {} + \int_{t_i}^{t_{i+1}} \mathopen{\exp} \left({\varphi(|\pi_\discrete| \omega) (t_{i+1} - \floor{\tau}_{\pi_\discrete} ) \omega} \right) \bracket{ u_\discrete ( \ceiling{\tau}_{\pi_\discrete} ) - u , f_\discrete (\tau) - v } \diff\tau .
\end{align*}
and together with the induction hypothesis we get
\begin{align*}
	\Xnorm{u_\discrete(t_{i+1}) - u} \le {} & \mathopen{\exp} \left({\varphi(|\pi_\discrete| \omega) t_{i+1} \omega} \right) \xnorm{u_\discrete^0 - u} \\*
 & + \int_0^{t_{i+1}} \mathopen{\exp} \left({\varphi(|\pi_\discrete| \omega) (t_{i+1} - \floor{\tau}_{\pi_\discrete} ) \omega} \right) \bracket{ u_\discrete ( \ceiling{\tau}_{\pi_\discrete} ) - u , f_\discrete (\tau) - v } \diff\tau .
\end{align*}
This completes the induction. The inequality \Cref{induktionsanfang_j} can be shown analogously.
\end{proof}

We can now already establish an upper bound for $a_{i, j}$, if the partitions are equidistant.

\begin{lemma} \label{equidistant-time}
    Let $N \in \N_{> T\omega}$. If $\pi_\discrete = \pi_{\hat\discrete} = \set{0, \frac{1}{N}T, \frac{2}{N}T, \ldots, \frac{N-1}{N}T, T}$, then for every $(u, v) \in A$ and for every $i \in \setsmall{0, \ldots, N}$ and $j \in \setsmall{0, \ldots, N}$ we have
    \begin{align*}
        a_{i, j} \le{}& \exp({\varphi(|\pi_{\discrete}| \omega) t_i \omega}) \xnorm{u_\discrete^0-u} + \exp({\varphi(|\pi_{\hat\discrete}| \omega) \hat t_j \omega}) \xnorm{u_{\hat\discrete}^0-u} \\*
        &+ \int_0^{(t_i-\hat t_j)^+} \mathopen{\exp}\klammern{\varphi(|\pi_{\discrete}| \omega) (t_i - \floor{\tau}_{\pi_\discrete}) \omega} \bracket{u_{\discrete} (\ceiling{\tau}_{\pi_{\discrete}}) - u, f_{\discrete} (\tau) - v} \diff\tau \\*
        &+ \int_0^{(\hat t_j-t_i)^+} \exp({\varphi(|\pi_{\hat\discrete}| \omega) (\hat t_j - \floor{\hat\tau}_{\pi_{\hat\discrete}}) \omega}) [u_{\hat\discrete} (\ceiling{\hat\tau}_{\pi_{\hat\discrete}}) - u, f_{\hat\discrete} (\hat\tau) - v] \diff\hat\tau \\*
        &+ \int_0^{t_i \land \hat t_j} \exp({\varphi(|\pi_{\discrete}| \omega) ((t_i\land\hat t_j) - \floor{\tau}_{\pi_\discrete}) \omega}) \Big[ \mathopen{u_{\discrete}} \left(\ceiling{\tau}_{\pi_{\discrete}} + (t_i-\hat t_j)^+\right) \\*
        & \hphantom{+ \int_0^{t_i \land \hat t_j}} - \mathopen{u_{\hat\discrete}} \left( \ceiling{\tau}_{\pi_{\hat\discrete}} + (\hat t_j-t_i)^+ \right), \mathopen{f_{\discrete}} \klammern{\tau+(t_i-\hat t_j)^+} - \mathopen{f_{\hat\discrete}} \klammern{\tau+(\hat t_j-t_i)^+} \Big] \diff\tau .
    \end{align*}
\end{lemma}

\begin{proof}
    We prove this by induction over $(i, j)$. Note that $|\pi_\discrete| \omega = \frac{T}{N} \omega < 1$. If $j = 0$, we can apply \Cref{bessereOmegaAbschaetzung} and use \Cref{induktionsanfang_i} to get
    \begin{align*}
        a_{i, 0} \le {} & \xnorm{u_\discrete(t_i) - u} + \xnorm{u_{\hat\discrete}^0 - u} \\
        \le {} & \mathopen{\exp} \left({\varphi(|\pi_\discrete| \omega) t_i \omega} \right) \xnorm{u_\discrete^0 - u} + \xnorm{u_{\hat\discrete}^0 - u} \\*
 & + \int_0^{t_i} \mathopen{\exp} \left({\varphi(|\pi_\discrete| \omega) (t_i - \floor{\tau}_{\pi_\discrete} ) \omega} \right) \bracket{ u_\discrete ( \ceiling{\tau}_{\pi_\discrete} ) - u , f_\discrete (\tau) - v } \diff\tau .
    \end{align*}
    If $i = 0$, we can use \Cref{induktionsanfang_j} analogously. If the statement is true for $(i, j)$, then note that $h_i = h_j = |\pi_\discrete| = |\pi_{\hat\discrete}|$, so by applying \Cref{bessereOmegaAbschaetzung}, we get
    \begin{align*}
        a_{i+1, j+1} \le {} & \exp ({ \varphi((|\pi_\discrete|)\omega) h_i \omega }) \left( a_{i,j} + h_i \bracket{u_\discrete (t_{i+1}) - u_{\hat{\discrete}} (\hat{t}_{j+1}) , f_\discrete (t_i) - f_{\hat{\discrete}} (\hat{t}_j)} \right) .
    \end{align*}
    By using
    \begin{align*}
        & \exp ({ \varphi((|\pi_\discrete|)\omega) h_i \omega }) h_i \bracket{u_\discrete (t_{i+1}) - u_{\hat{\discrete}} (\hat{t}_{j+1}) , f_\discrete (t_i) - f_{\hat{\discrete}} (\hat{t}_j)} \\
        {}={} & \int_{t_i \land \hat t_j}^{t_{i+1} \land \hat t_{j+1}} \exp({\varphi(|\pi_{\discrete}| \omega) ((t_{i+1}\land\hat t_{j+1}) - \floor{\tau}_{\pi_\discrete}) \omega}) \Big[ u_{\discrete} \left(\ceiling{\tau}_{\pi_{\discrete}} + (t_{i+1}-\hat t_{j+1})^+\right) \\*
        & - u_{\hat\discrete} \left( \ceiling{\tau}_{\pi_{\hat\discrete}} + (\hat t_{j+1}-t_{i+1})^+ \right), f_{\discrete} \klammern{\tau+(t_{i+1}-\hat t_{j+1})^+} - f_{\hat\discrete} \klammern{\tau+(\hat t_{j+1}-t_{i+1})^+} \Big] \diff\tau .
    \end{align*}
    and the induction hypothesis for $a_{i, j}$ as well as
    \begin{align*}
        (t_i - \hat t_j)^+ &= (t_{i+1}-\hat t_{j+1})^+ \text{ and} \\
        (\hat t_j - t_i)^+ &= (\hat t_{j+1}-t_{i+1})^+,
    \end{align*}
    we get the statement for $(i+1, j+1)$. This completes the induction.
\end{proof}

To get an upper bound for $a_{i, j}$, which holds for all partitions $\pi_\discrete$ and $\pi_{\hat\discrete}$, we use the following density function $\rho^{i, j}$.

\begin{definition} \label{definition-of-rho}
Let $\rho^{i,j} \colon \R^2 \to \R$ be recursively defined by
\[ \rho^{i,0} \coloneqq \mathbf{1}_{[0, t_i) \times [-1, 0)}  \numberthis \label{rho_definition_1} \]
for $i \in \set{0,\ldots,N}$,
\[ \rho^{0,j} \coloneqq \mathbf{1}_{[-1, 0) \times [0, \hat t_j)} \numberthis \label{rho_definition_2} \]
for $j \in \setsmall{0,\ldots,\hat{N}}$ and
\begin{align*}
    \rho^{i+1,j+1} \coloneqq {}&\bigg(1- \frac{\hat{h}_j}{h_i}\bigg)^+ \rho^{i+1,j} + \bigg(1- \frac{h_i}{\hat{h}_j}\bigg)^+ \rho^{i,j+1} + \frac{h_i \land \hat{h}_j}{h_i \lor \hat{h}_j} \rho^{i,j}  \\*
    &+ \frac{1}{h_i \lor \hat{h}_j} \mathbf{1}_{[t_i , t_{i+1}) \times [\hat{t}_j , \hat{t}_{j+1})} \numberthis \label{rho_definition_3}
\end{align*}
for $i \in \set{0,\ldots,N-1}$ and $j \in \setsmall{0,\ldots,\hat{N}-1}$.
\end{definition}

\begin{figure}
    \centering
    % This file was created with tikzplotlib v0.10.1.
\begin{tikzpicture}

\definecolor{darkgray176}{RGB}{176,176,176}

\begin{axis}[
colorbar,
colorbar style={ylabel={}},
colormap/viridis,
point meta max=2,
point meta min=0,
tick align=outside,
tick pos=left,
x grid style={darkgray176},
xmin=-1, xmax=4,
xtick style={color=black},
xtick={-1,0,4},
xticklabels = {-1,0,$t_i$},
y grid style={darkgray176},
ymin=-1, ymax=3,
ytick style={color=black},
ytick={-1,0,3},
yticklabels={-1,0,$\hat t_j$},
]
\addplot graphics [includegraphics cmd=\pgfimage,xmin=-1, xmax=4, ymin=-1, ymax=3] {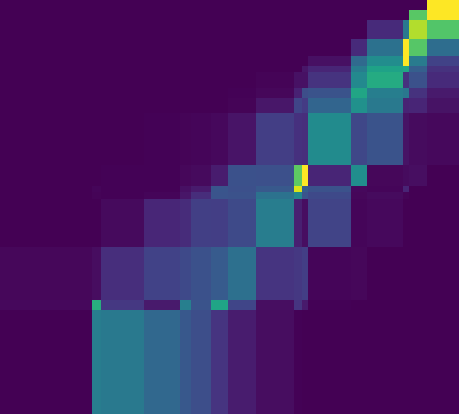};
\end{axis}

\end{tikzpicture}
    % This file was created with tikzplotlib v0.10.1.
\begin{tikzpicture}

\definecolor{darkgray176}{RGB}{176,176,176}

\begin{axis}[
colorbar,
colorbar style={ylabel={}},
colormap/viridis,
point meta max=2,
point meta min=0,
tick align=outside,
tick pos=left,
% title={\(\displaystyle \rho^{i,j} 
% % (\tau, \hat \tau)
% \)},
x grid style={darkgray176},
xmin=-1, xmax=4,
xtick style={color=black},
xtick={-1,0,4},
xticklabels={-1,0,$t_i$},
y grid style={darkgray176},
ymin=-1, ymax=3,
ytick style={color=black},
ytick={-1,0,3},
yticklabels={-1,0,$\hat t_j$},
]
\addplot graphics [includegraphics cmd=\pgfimage,xmin=-1, xmax=4, ymin=-1, ymax=3] {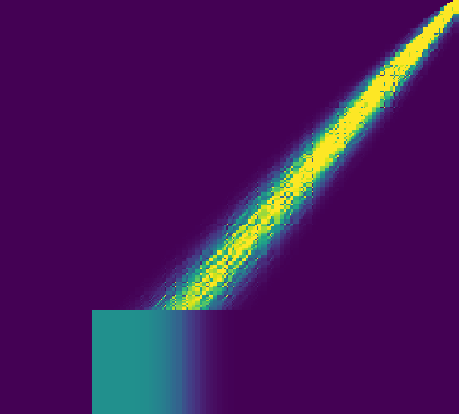};
\end{axis}

\end{tikzpicture}
    \caption{Plots of the density $\rho^{i, j}$ for different partitions $\pi_\discrete$ and $\pi_{\hat\discrete}$. For small mesh sizes $|\pi_\discrete|$ and $|\pi_{\hat\discrete}|$ the density $\rho^{i, j}$ is concentrated along the diagonal through the point $(t_i, \hat t_j)$.}
    \label{fig:rho-plots}
\end{figure}

\begin{lemma}\label{implicit-upper-bound}
	Let $(u,v) \in A$ and $g \in BV( -1, T; X)$ with $g (\tau) = v$ for $\tau < 0$. If $(|\pi_\discrete| \lor |\pi_{\hat\discrete}|) \omega < 1$, then for all $i \in \set{0, \ldots N}$ and $j \in \setsmall{0, \ldots, \hat N}$ we have
	\begin{align*}
		a_{i, j} \le {} & \mathopen{\exp}\klammern{\varphi((|\pi_\discrete| \lor |\pi_{\hat\discrete}|)\omega) (t_i+\hat t_j) \omega^+} \bigg( \xnorm{u_\discrete^0 - u} + \xnorm{u_{\hat\discrete}^0 - u} \\
		& {} + \norm{f_\discrete - g}_{L^1(0,t_i;X)} + \norm{f_{\hat\discrete} - g}_{L^1(0,\hat t_j;X)} + \int_{-1}^T \int_{-1}^T \rho^{i,j} (\tau, \hat\tau) \Xnorm{g(\tau) - g(\hat \tau)} \diff\hat\tau \diff\tau \bigg) .
	\end{align*}
\end{lemma}

\begin{proof}
We prove this by induction over $(i,j)$. If $j = 0$, we can use \Cref{induktionsanfang-v2} to get
\begin{align*}
	a_{i, 0} = {} & \xnorm{u_\discrete(t_i) - u_{\hat\discrete}^0} \\*
    \le {} & \Xnorm{u_\discrete(t_i) - u} + \xnorm{u_{\hat\discrete}^0-u} \\*
    \le {} & \mathopen{\exp} \left({\varphi(|\pi_\discrete| \omega) t_i \omega} \right) \xnorm{u_\discrete^0 - u} + \xnorm{u_{\hat\discrete}^0 - u} \\*
 & + \int_0^{t_i} \mathopen{\exp} \left({\varphi(|\pi_\discrete| \omega) (t_i - \floor{\tau}_{\pi_\discrete} ) \omega} \right) \bracket{ u_\discrete ( \ceiling{\tau}_{\pi_\discrete} ) - u , f_\discrete (\tau) - v } \diff\tau \\*
    \le {} & \mathopen{\exp}\klammern{\varphi((|\pi_\discrete| \lor |\pi_{\hat\discrete}|)\omega) t_i \omega^+} \bigg( \xnorm{u_\discrete^0 - u} + \xnorm{u_{\hat\discrete}^0 - u} + \int_0^{t_i} \Xnorm{ f_\discrete (\tau) - v } \diff\tau \bigg) \\*
	\le {} & \mathopen{\exp}\klammern{\varphi((|\pi_\discrete| \lor |\pi_{\hat\discrete}|)\omega) t_i \omega^+} \bigg( \xnorm{u_\discrete^0 - u} + \xnorm{u_{\hat\discrete}^0 - u} \\*
	& + \norm{f_\discrete - g}_{L^1(0,t_i;X)} + \int_{-1}^{T} \int_{-1}^{T} \rho^{i,0} (\tau, \hat\tau) \Xnorm{g(\tau) - g(\hat \tau)} \diff\hat\tau \diff\tau \bigg).
\end{align*}
The case $i = 0$ can be treated analogously. If we assume the statement to be true for $(i,j)$, $(i+1,j)$ and $(i, j+1)$, then we can use \Cref{bessereOmegaAbschaetzung} to get
\begin{align*}
	a_{i+1,j+1} \le {} & \mathopen{\exp} \left( \varphi((|\pi_\discrete|\lor|\pi_{\hat{\discrete}}|)\omega) (h_i \land \hat{h}_j) \omega^+ \right) \Bigg( \bigg(1- \frac{\hat{h}_j}{h_i}\bigg)^+ a_{i+1,j} + \bigg(1- \frac{h_i}{\hat{h}_j}\bigg)^+ a_{i,j+1} \\ & + \frac{h_i \land \hat{h}_j}{h_i \lor \hat{h}_j} a_{i,j} + (h_i \land \hat{h}_j) \bracket{u_\discrete (t_{i+1}) - u_{\hat{\discrete}} (\hat{t}_{j+1}) , f_\discrete (t_i) - f_{\hat{\discrete}} (\hat{t}_j)} \Bigg) .
\end{align*}
Note that
\begin{align*}
	&\phantom{{}={}} (h_i \land \hat{h}_j) \bracket{u_\discrete (t_{i+1}) - u_{\hat{\discrete}} (\hat{t}_{j+1}) , f_\discrete (t_i) - f_{\hat{\discrete}} (\hat{t}_j)} \\
	&\le (h_i \land \hat h_j) \xnorm{f_\discrete (t_i) - f_{\hat\discrete}(\hat t_j)} \\
	&= \frac{1}{h_i \lor \hat h_j} \int_{t_i}^{t_{i+1}} \int_{\hat t_j}^{\hat t_{j+1}} \Xnorm{f_\discrete (\tau) - g(\tau) + g(\tau) - g(\hat\tau) + g(\hat\tau) - f_{\hat\discrete}(\hat\tau)} \diff\hat\tau \diff\tau \\
	&\le \frac{\hat h_j}{h_i \lor \hat h_j} \norm{f_\discrete-g}_{L^1(t_i,t_{i+1};X)} + \frac{h_i}{hi \lor \hat h_j} \norm{f_{\hat\discrete}-g}_{L^1(\hat t_j,\hat t_{j+1};X)} \\*
	&\phantom{{}={}}+ \frac{1}{h_i \lor \hat h_j} \int_{-1}^{T} \int_{-1}^{T} \mathbf{1}_{[t_i , t_{i+1}) \times [\hat{t}_j , \hat{t}_{j+1})} (\tau, \hat\tau) \Xnorm{g(\tau) - g(\hat\tau)} \diff\hat\tau \diff\tau .
\end{align*}
By using the equalities
\begin{align*}
    1 &= \klammern{1- \frac{\hat{h}_j}{h_i}}^+ + \klammern{1- \frac{h_i}{\hat{h}_j}}^+ + \frac{h_i \land \hat{h}_j}{h_i \lor \hat{h}_j} \\
    &= \klammern{1- \frac{\hat{h}_j}{h_i}}^+ + \frac{\hat h_j}{h_i \lor \hat h_j} \\
    &= \klammern{1-\frac{h_i}{\hat h_j}}^+ + \frac{h_i}{h_i \lor \hat h_j}
\end{align*}
as well as the recursive definition of $\rho^{i+1, j+1}$ in \Cref{rho_definition_3} we get
\begin{align*}
	a_{i+1,j+1} \le {} & \mathopen{\exp}\klammern{\varphi((|\pi_\discrete| \lor |\pi_{\hat\discrete}|)\omega) (t_{i+1}+\hat t_{j+1}) \omega^+} \bigg( \xnorm{u_\discrete^0 - u} + \xnorm{u_{\hat\discrete}^0 - u} \\*
		& {} + \norm{f_\discrete - g}_{L^1(0,t_{i+1};X)} + \norm{f_{\hat\discrete} - g}_{L^1(0,\hat t_{j+1};X)} \\*
  & + \int_{-1}^T \int_{-1}^T \rho^{i+1,j+1} (\tau, \hat\tau) \Xnorm{g(\tau) - g(\hat \tau)} \diff\hat\tau \diff\tau \bigg) .
\end{align*}
	This completes the induction.
\end{proof}

\Cref{implicit-upper-bound} already achieves the goal of finding an upper bound for $a_{i, j}$, but it is not clear, whether the upper bound is small, if the mesh sizes $|\pi_\discrete|$ and $|\pi_{\hat\discrete}|$ are small. 
The next results show some properties of the density $\rho^{i,j}$ and lead to a good explicit estimate of the double integral
\[ \int_{-1}^T \int_{-1}^T \rho^{i,j} (\tau, \hat\tau) \Xnorm{g(\tau) - g(\hat \tau)} \diff\hat\tau \diff\tau . \]

\subsection{Properties of the density} \label{section-2-2-properties-of-the-density}

\begin{lemma} \label{rho-support}
For $i \in \set{0, \ldots, N}$ and $j \in \setsmall{0, \ldots, \hat N}$ we have
\[ \supp \rho^{i,j} \subseteq ( [-1, t_i] \times [-1, \hat t_j] ) \setminus ( [-1, 0) \times [-1, 0) ) . \]
\end{lemma}

\begin{proof}
	This is a direct consequence of the definition of the density $\rho^{i, j}$. We prove this by induction over $(i,j)$. For $(i, j) = (0, 0)$ the statement is true since $\supp \rho^{0,0} = \emptyset$. For $i \in \set{1, \ldots, N}$ we have 
 \[ \supp \rho^{i,0} = [0,t_i] \times [-1,0] = ([-1, t_i] \times [-1, \hat t_0]) \setminus ( [-1, 0) \times [-1, 0) ) \]
 and for $j \in \setsmall{1, \ldots, \hat N}$ we have
 \[ \supp \rho^{0,j} = [-1,0] \times [0, \hat t_j] = ( [-1, t_0] \times [-1, \hat t_j] ) \setminus ( [-1, 0) \times [-1, 0) ) . \]
 If the statement is true for $(i, j)$, $(i+1, j)$ and $(i, j+1)$, then
 \begin{align*}
     \supp \rho^{i+1, j+1} &\subseteq \supp \rho^{i+1, j} \cup \supp \rho^{i, j+1} \cup \supp \rho^{i, j} \cup ( [t_i, t_{i+1}] \times [\hat t, \hat t_{i+1}] ) \\
    &\subseteq ( [-1, t_{i+1}] \times [-1, \hat t_{j+1}] ) \setminus ([-1, 0) \times [-1, 0)) .
 \end{align*}
 This concludes the induction.
\end{proof}

\begin{lemma}\label{mass-preservation}
Let $i \in \set{0,\ldots,N}$ and $j \in \setsmall{0,\ldots,\hat{N}}$. For all $\tau \ge 0$ we have
\[ \int_{-1}^{T} \rho^{i, j} (\tau, \hat \tau ) \diff\hat\tau = \mathbf{1}_{[0, t_i)} ( \tau ) , \]
and for all $\hat \tau \ge 0$ we have
\[ \int_{-1}^{T} \rho^{i, j} (\tau , \hat \tau ) \diff\tau = \mathbf{1}_{[0, \hat t_j)}(\hat \tau) . \]
\end{lemma}

\begin{proof}
	We prove the first statement by induction over $(i,j)$. Let $\tau \ge 0$. If $j = 0$, then
	\[ \int_{-1}^{T} \rho^{i, 0} (\tau, \hat \tau ) \diff\hat\tau = \int_{-1}^{T} \mathbf{1}_{[0, t_i)} (\tau) \mathbf{1}_{[-1,0)} (\hat \tau ) \diff\hat\tau = \mathbf{1}_{[0, t_i)} (\tau) . \]
	If $i = 0$, then
	\[ \int_{-1}^{T} \rho^{0, j} (\tau, \hat \tau ) \diff\hat\tau = \int_{-1}^{T} \mathbf{1}_{[-1, 0)} (\tau) \mathbf{1}_{[0,\hat t_j)} (\hat \tau ) \diff\hat\tau = 0 =\mathbf{1}_{[0, t_0)} (\tau) . \]
	If the statement is true for $(i, j)$, $(i+1, j)$ and $(i, j+1)$, then together with the recursive definition of $\rho^{i+1, j+1}$ in \Cref{rho_definition_3} we get
	\begin{align*}
		\int_{-1}^{T} \rho^{i+1, j+1} (\tau, \hat \tau ) \diff\hat\tau = {} & \bigg(1- \frac{\hat{h}_j}{h_i}\bigg)^+ \mathbf{1}_{[0,t_{i+1})} (\tau) + \bigg(1- \frac{h_i}{\hat{h}_j}\bigg)^+ \mathbf{1}_{[0,t_i)} (\tau) \\*
		&{} + \frac{h_i \land \hat{h}_j}{h_i \lor \hat{h}_j} \mathbf{1}_{[0,t_i)} (\tau)  + \frac{1}{h_i \lor \hat{h}_j} \underbrace{ \int_{-1}^T \mathbf{1}_{[t_i , t_{i+1})} (\tau) \mathbf{1}_{[\hat{t}_j , \hat{t}_{j+1})} (\hat\tau) \diff\hat\tau }_{{}={}\hat h_j \mathbf{1}_{[t_i, t_{i+1})}(\tau)} \\*
		{}= {}& \mathbf{1}_{[0, t_{i+1})} (\tau).
	\end{align*}
	This concludes the induction. The second statement can be treated analogously.
\end{proof}

\begin{remark} \label{mass-preservation-remark}
\Cref{mass-preservation} shows that the mass stays constant in some sense. More precisely, it follows from \Cref{mass-preservation}, that
\[ \int_0^{T} \int_{-1}^{T} \rho^{i, j} (\tau , \hat \tau ) \diff\hat\tau \diff\tau = \int_0^{T} \mathbf{1}_{[0,t_i)} (\tau) \diff\tau = t_i \]
and
\[ \int_{-1}^{T} \int_0^{T} \rho^{i, j} (\tau , \hat \tau ) \diff\hat\tau \diff\tau = \int_0^{T} \mathbf{1}_{[0,\hat t_j)} (\hat\tau) \diff\hat\tau = \hat t_j. \]
Furthermore one can show that
\[ \int_{-1}^{T} \rho^{i, j} (\tau, \hat \tau ) \diff\hat\tau \le \hat t_j \text{ for all $\tau < 0$} \quad \text{and} \quad \int_{-1}^{T} \rho^{i, j} (\tau , \hat \tau ) \diff\tau \le t_i \text{ for all $\hat \tau < 0$}. \]
Therefore
\[ \int_0^{T} \int_0^{T} \rho^{i, j} (\tau , \hat \tau ) \diff\hat\tau \diff\tau \le t_i \land \hat t_j \le t_i \lor \hat t_j \le \int_{-1}^{T} \int_{-1}^{T} \rho^{i, j} (\tau , \hat \tau ) \diff\hat\tau \diff\tau \le t_i + \hat t_j . \]
\end{remark}

The next lemma shows how the density $\rho^{i,j}$ can be computed directly without iteration over $i$ and $j$.

\begin{lemma}
	For every $i \in \set{0,\ldots,N}$, $j \in \setsmall{0,\ldots,\hat{N}}$, $k \in \set{0,\ldots,N-1}$ and every $l \in \setsmall{0,\ldots,\hat N-1}$,
	\begin{align*}
		\rho^{i,j} (t_k, \hat t_l) = {}  \frac{1}{h_k \lor \hat{h}_l} \bigg( &(h_k - \hat{h}_{l+1})^+ \rho^{i,j}(t_k,\hat t_{l+1}) + (\hat{h}_l - h_{k+1})^+ \rho^{i,j} (t_{k+1},\hat t_l) \\*
		&+ (h_{k+1} \land \hat{h}_{l+1}) \rho^{i,j} (t_{k+1},\hat t_{l+1})
		+ \delta_{ (i, k+1), (j, l+1) }
		\bigg).
		\label{rho_direkt}
	\end{align*}
\end{lemma}

\begin{proof}
In this proof we use the short notation $\rho^{i,j}_{k,l} \coloneqq \rho^{i,j}(t_k, \hat t_l)$ and the Kronecker delta.
Let $k \in \set{0,\ldots,N-1}$ and $l \in \setsmall{0,\ldots,\hat N-1}$.
We prove the statement by induction over $(i,j)$. For $i=0$ or $j = 0$ this is true since
\[ \rho^{i,0} (\tau, \hat\tau) = \rho^{0,j} (\tau, \hat\tau) = 0 \]
for all $\tau, \hat\tau \ge 0$ by definition (see \Cref{rho_definition_1} and \Cref{rho_definition_2}) and
\[ \mathbf{1}_{[t_{i} , t_{i+1}) \times [\hat{t}_{0} , \hat{t}_{1})} (t_{k+1}, \hat t_{l+1}) = \mathbf{1}_{[t_{0} , t_{1}) \times [\hat{t}_{j} , \hat{t}_{j+1})} (t_{k+1}, \hat t_{l+1}) = 0 . \]
Now let us assume that the statement is true for $(i+1,j)$, $(i,j+1)$ and $(i,j)$. Using the definition of $\rho^{i+1, j+1}$ in \Cref{rho_definition_3} and our induction hypothesis we get
\begin{align*}
	\rho^{i+1,j+1}_{k,l}
	= {}&\bigg(1- \frac{\hat{h}_j}{h_i}\bigg)^+ \rho^{i+1,j}_{k,l} + \bigg(1- \frac{h_i}{\hat{h}_j}\bigg)^+ \rho^{i,j+1}_{k,l}
	+ \frac{h_i \land \hat{h}_j}{h_i \lor \hat{h}_j} \rho^{i,j}_{k,l} + \frac{1}{h_i \lor \hat{h}_j} \delta_{(i,j), (k,l)} \\
	{}={}& \bigg(1- \frac{\hat{h}_j}{h_i}\bigg)^+ \frac{1}{h_k \lor \hat h_l} \Big( (h_k - \hat{h}_{l+1})^+ \rho^{i+1,j}_{k,l+1} + (\hat{h}_l - h_{k+1})^+ \rho^{i+1,j}_{k+1,l} \\*& + (h_{k+1} \land \hat{h}_{l+1}) \rho^{i+1,j}_{k+1,l+1} + \delta_{(i+1,j),(k+1,l+1)} \Big) \\&
    + \bigg(1- \frac{h_i}{\hat{h}_j}\bigg)^+ \frac{1}{h_k \lor \hat h_l} \Big( (h_k - \hat{h}_{l+1})^+ \rho^{i,j+1}_{k,l+1} + (\hat{h}_l - h_{k+1})^+ \rho^{i,j+1}_{k+1,l} \\*& + (h_{k+1} \land \hat{h}_{l+1}) \rho^{i,j+1}_{k+1,l+1} + \delta_{(i,j+1),(k+1,l+1)} \Big) \\&
    + \frac{h_i \land \hat{h}_j}{h_i \lor \hat{h}_j} \frac{1}{h_k \lor \hat h_l} \Big( (h_k - \hat{h}_{l+1})^+ \rho^{i,j}_{k,l+1} + (\hat{h}_l - h_{k+1})^+ \rho^{i,j}_{k+1,l} \\*& + (h_{k+1} \land \hat{h}_{l+1}) \rho^{i,j}_{k+1,l+1} + \delta_{(i,j),(k+1,l+1)} \Big) 
    + \frac{1}{h_i \lor \hat{h}_j} \delta_{(i,j), (k,l)}.
\end{align*}
By using the identities
\begin{alignat*}{2}
	\bigg(1- \frac{\hat{h}_j}{h_i}\bigg)^+ \delta_{(i+1,j), (k+1, l+1)} &= \frac{(h_i-\hat h_j)^+}{h_i \lor \hat h_j} \delta_{(i+1,j), (k+1, l+1)} &&= \frac{(h_k - \hat h_{l+1})^+}{h_i \lor \hat h_j} \delta_{(i,j),(k,l+1)} , \\
	\bigg(1- \frac{h_i}{\hat{h}_j}\bigg)^+ \delta_{(i,j+1),(k+1,l+1)} &= \frac{(\hat h_j - h_i)^+}{h_i \lor \hat h_j} \delta_{(i,j+1),(k+1,l+1)} &&= \frac{(\hat h_l - h_{k+1})^+}{h_i \lor \hat h_j} \delta_{(i,j),(k+1,l)} , \\
	\frac{h_i \land \hat{h}_j}{h_i \lor \hat{h}_j} \delta_{(i,j),(k+1,l+1)} &= \frac{h_{k+1} \land \hat h_{l+1}}{h_i \lor \hat h_j} \delta_{(i,j),(k+1,l+1)} &&\quad \text{and} \\
	\frac{1}{h_i \lor \hat{h}_j} \delta_{(i,j), (k,l)} &= \frac{1}{h_k \lor \hat{h}_l} \delta_{(i+1,j+1), (k+1,l+1)}
\end{alignat*}
and rearranging the terms we get
\begin{align*}
	&\rho^{i+1,j+1}_{k,l} \\ = {} & \frac{(h_k-\hat h_{l+1})^+}{h_k \lor \hat h_l} \bigg( \bigg(1-\frac{\hat h_j}{h_i} \bigg)^+ \rho^{i+1,j}_{k,l+1} + \bigg(1-\frac{h_i}{\hat h_j}\bigg)^+ \rho^{i,j+1}_{k,l+1} + \frac{h_i \land \hat{h}_j}{h_i \lor \hat{h}_j} \rho^{i,j}_{k,l+1} + \frac{ \delta_{(i,j),(k,l+1)} }{h_i \lor \hat h_j} \bigg) \\*
	& + \frac{(\hat h_l - h_{k+1})^+}{h_k \lor \hat h_l} \bigg( \bigg(1-\frac{\hat h_j}{h_i} \bigg)^+ \rho^{i+1,j}_{k+1,l} + \bigg(1-\frac{h_i}{\hat h_j}\bigg)^+ \rho^{i,j+1}_{k+1,l} + \frac{h_i \land \hat{h}_j}{h_i \lor \hat{h}_j} \rho^{i,j}_{k+1,l} + \frac{ \delta_{(i,j),(k+1,l)} }{h_i \lor \hat h_j} \bigg) \\*
	& + \frac{h_{k+1} \land \hat{h}_{l+1}}{h_k \lor \hat{h}_l} \bigg( \bigg(1-\frac{\hat h_j}{h_i} \bigg)^+ \rho^{i+1,j}_{k+1,l+1} + \bigg(1-\frac{h_i}{\hat h_j}\bigg)^+ \rho^{i,j+1}_{k+1,l+1} + \frac{h_i \land \hat{h}_j}{h_i \lor \hat{h}_j} \rho^{i,j}_{k+1,l+1} \\*& + \frac{ \delta_{(i,j),(k+1,l+1)} }{h_i \lor \hat h_j} \bigg)
	+ \frac{1}{h_k \lor \hat{h}_l} \delta_{(i+1,j+1), (k+1,l+1)} \\
	{}={}& \frac{1}{h_k \lor \hat h_l} \Big( (h_k - \hat{h}_{l+1})^+ \rho^{i+1,j+1}_{k,l+1} + (\hat{h}_l - h_{k+1})^+ \rho^{i+1,j+1}_{k+1,l} \\*
 &+ (h_{k+1} \land \hat{h}_{l+1}) \rho^{i+1,j+1}_{k+1,l+1} + \delta_{(i+1,j+1),(k+1,l+1)} \Big) .
\end{align*}
This concludes the induction.
\end{proof}

We only need one more technical lemma before we can state the main estimate for $\rho^{i, j}$.

\begin{lemma}\label{abc-lemma}
	Let $a, b, c > 0$. Then \[ a+b-\frac{2ab}{c} \le \sqrt{(a-b)^2+ac} . \]
\end{lemma}

\begin{proof}
If $4b \le c$, then we have $a+b-\frac{2ab}{c} \le a + b = \sqrt{(a-b)^2+4ab} \le \sqrt{(a-b)^2+ac}$.

If $4b > c$, then we consider the function $h \colon [0, \infty) \to \R , a \mapsto \sqrt{(a-b)^2+ac} - a-b+\frac{2ab}{c}$ and get
\begin{align*}
h'(a) &= \frac{2(a-b) + c}{2\sqrt{(a-b)^2+ac}} + \frac{2b}{c} - 1 \quad \text{and} \\
h''(a) &= \frac{c(4b-c)}{4\klammern{(a-b)^2+ac}^{\frac{3}{2}}}
\end{align*}
for $a \ge 0$. Since $h(0) = 0$, $h'(0) = \frac{c}{2b} + \frac{2b}{c} - 2 = \klammern{\sqrt{\frac{c}{2b}} - \sqrt{\frac{2b}{c}}}^2\ge 0$ and $h''(a) \ge 0$ for all $a \ge 0$, we get $h(a) \ge 0$ for all $a \ge 0$, which proves the statement.
\end{proof}

The following lemma shows the concentration of mass on the diagonal.

\begin{lemma}\label{mass-concentration}
Let $t  \in \R$. For every $i \in \set{0, \ldots, N}$ and $j \in \setsmall{0, \ldots, \hat N}$ we  have
\begin{align*}
    & \int_{-1}^t \int_t^T \rho^{i,j} (\tau, \hat\tau) \diff\hat\tau \diff\tau + \int_t^T \int_{-1}^t \rho^{i,j} (\tau, \hat\tau) \diff\hat\tau \diff\tau \\*
    {}\le{}& \sqrt{(t_i - \hat t_j)^2 + |\pi_\discrete| ( t_i \land (t_i - t)^+) + |\pi_{\hat\discrete}| (\hat t_j \land (\hat t_j - t)^+)} .
\end{align*}
\end{lemma}

\begin{proof}
For this proof we use the notation
\[ \kappa_{i, j} (t) \coloneqq \int_{-1}^t \int_t^T \rho^{i,j} (\tau, \hat\tau) \diff\hat\tau \diff\tau + \int_t^T \int_{-1}^t \rho^{i,j} (\tau, \hat\tau) \diff\hat\tau \diff\tau .
	\numberthis \label{kappa-definition}
\]
First, let $t \ge 0$. Note that by \Cref{mass-preservation}, for every $i \in \set{0, \ldots, N}$ and $j \in \setsmall{0, \ldots, \hat N}$ we have
\[ \kappa_{i, j} (t) \le \int_t^T \mathbf{1}_{[0,\hat t_j)} (\hat\tau) \diff\hat\tau + \int_{t}^{T} \mathbf{1}_{[0,t_i)} (\tau) \diff\tau = (t_i - t)^+ + (\hat t_j - t)^+ . \numberthis \label{kappa-by-mass-preservation} \]
In particular $\kappa_{i,j} (t) = 0$ for $t \ge t_i \lor \hat t_j$.
We now prove the statement by induction over $(i,j)$. If $j = 0$, then we have
\[ \kappa_{i,0} (t) = 0 + (t_i-t)^+ \le t_i = \sqrt{(t_i-\hat t_0)^2} .  \]
If $i = 0$, we have
\[ \kappa_{0,j} (t) = (\hat t_j-t)^+ + 0 \le \hat t_j = \sqrt{(t_0 - \hat t_j)^2} . \]
Now let us assume that the statement is true for $(i,j)$, $(i+1,j)$ and $(i, j+1)$. We first treat the case that $\hat h_j \le h_i$. Then by the recursive definition of $\rho^{i+1, j+1}$ in \Cref{rho_definition_3} we have
\begin{align*}
    &\kappa_{i+1,j+1}(t) \\* = {} & \bigg(1- \frac{\hat{h}_j}{h_i}\bigg) \kappa_{i+1, j}(t) + \frac{\hat{h}_j}{h_i} \kappa_{i, j}(t) \\*
    & + \frac{1}{h_i} \Bigg( \int_{-1}^t \int_t^T \mathbf{1}_{[t_i , t_{i+1})} (\tau) \mathbf{1}_{[\hat{t}_j , \hat{t}_{j+1})} (\hat\tau) \diff\hat\tau \diff\tau + \int_t^T \int_{-1}^t \mathbf{1}_{[t_i , t_{i+1})} (\tau) \mathbf{1}_{[\hat{t}_j , \hat{t}_{j+1})} (\hat\tau) \diff\hat\tau \diff\tau \Bigg) \\
    ={}& \bigg(1- \frac{\hat{h}_j}{h_i}\bigg) \kappa_{i+1, j}(t) + \frac{\hat{h}_j}{h_i} \kappa_{i, j}(t) \\*
    & {} + \frac{1}{h_i} \Bigg( ((t-t_i)^+\land h_i)((\hat t_{j+1}-t)^+ \land \hat h_j)  + ((t_{i+1}-t)^+\land h_i)((t-\hat t_j)^+ \land \hat h_j) \Bigg) . \numberthis \label{kappa-computed}
\end{align*}
If $0 \le t \le t_i \land \hat t_j$, then we have $(t-t_i)^+ = (t-\hat t_j)^+ = 0$, so the last term in \Cref{kappa-computed} vanishes. Using also the induction hypothesis and the concavity of the square root function, we get
\begin{align*}
	\kappa_{i+1, j+1}(t) &= \bigg(1- \frac{\hat{h}_j}{h_i}\bigg) \kappa_{i+1, j}(t) + \frac{\hat{h}_j}{h_i} \kappa_{i, j}(t) \\
	&\le \bigg(1- \frac{\hat{h}_j}{h_i}\bigg) \sqrt{(t_{i+1} - \hat t_j)^2 + |\pi_\discrete| (t_{i+1} - t) + |\pi_{\hat\discrete}| (\hat t_j - t)} \\*
	&\phantom{{}={}} + \frac{\hat{h}_j}{h_i } \sqrt{(t_i - \hat t_j)^2 + |\pi_\discrete| (t_i - t) + |\pi_{\hat\discrete}| (\hat t_j - t)} \\
	&\le \bigg( \bigg(1- \frac{\hat{h}_j}{h_i}\bigg) \klammern{(t_{i+1} - \hat t_j)^2 + |\pi_\discrete| (t_{i+1} - t) + |\pi_{\hat\discrete}| (\hat t_j - t)} \\*
	&\phantom{{}={}} + \frac{\hat{h}_j}{h_i } \klammern{(t_i - \hat t_j)^2 + |\pi_\discrete| (t_i - t) + |\pi_{\hat\discrete}| (\hat t_j - t)} \bigg)^\frac{1}{2} \\
	&= \bigg( (t_{i+1} - \hat{t}_{j+1})^2 + |\pi_\discrete| (t_{i+1}-t) + |\pi_{\hat\discrete}| ( \hat{t}_{j+1} -t ) \\*
    &\phantom{{}={}}+ \bigg(1-\frac{\hat{h}_j}{h_i}\bigg) \big( 2(t_{i+1}-\hat{t}_{j+1})\hat{h}_j + \hat{h}_j^2 - |\pi_{\hat{\discrete}}| \hat{h}_{j} \big) \\*
    &\phantom{{}={}} + \frac{\hat{h}_j}{h_i} \big( 2(t_{i+1} - \hat{t}_{j+1})(\hat{h}_j - h_i) + \hat{h}_j^2 - 2\hat{h}_j h_i + h_i^2 - |\pi_\discrete| h_i - |\pi_{\hat{\discrete}}| \hat{h}_j \big) \bigg)^{\frac{1}{2}} \\
    &= \sqrt{ (t_{i+1} - \hat{t}_{j+1})^2 + |\pi_\discrete| (t_{i+1}-t) + |\pi_{\hat{\discrete}}| (\hat{t}_{j+1}-t) - \hat{h}_j ( |\pi_\discrete| - h_i ) - |\pi_{\hat{\discrete}}| \hat{h}_j - \hat{h}_j^2 } \\
    &\le \sqrt{ (t_{i+1} - \hat{t}_{j+1})^2 + |\pi_\discrete| (t_{i+1}-t) + |\pi_{\hat{\discrete}}| (\hat{t}_{j+1}-t) } .
\end{align*}
If $t_i < t < \hat t_j$, then we can use \Cref{kappa-by-mass-preservation}, \Cref{kappa-computed}, the induction hypothesis and the preceding calculation to get
\begin{align*}
	\kappa_{i+1,j+1}(t) &= \bigg(1- \frac{\hat{h}_j}{h_i}\bigg) \kappa_{i+1,j}(t) + \frac{\hat{h}_j}{h_i} \kappa_{i,j}(t) + \frac{1}{h_i} \klammern{(t-t_i) \land h_i} \hat h_j \Newline
	&\le \bigg(1- \frac{\hat{h}_j}{h_i}\bigg) \kappa_{i+1, j}(t) + \frac{\hat{h}_j}{h_i } (\hat t_j - t) + \frac{\hat h_j}{h_i} (t-t_i) \Newline
	&=  \bigg(1- \frac{\hat{h}_j}{h_i}\bigg) \kappa_{i+1, j}(t) + \frac{\hat{h}_j}{h_i } (\hat t_j -t_i) \\
	&\le \bigg(1- \frac{\hat{h}_j}{h_i}\bigg) \sqrt{(t_{i+1} - \hat t_j)^2 + |\pi_\discrete| (t_{i+1} - t) + |\pi_{\hat\discrete}| (\hat t_j - t)} \\*
	&\phantom{{}={}} + \frac{\hat{h}_j}{h_i } \sqrt{(t_i - \hat t_j)^2 + |\pi_\discrete| (t_i - t) + |\pi_{\hat\discrete}| (\hat t_j - t)} \\
	&\le \sqrt{ (t_{i+1} - \hat{t}_{j+1})^2 + |\pi_\discrete| (t_{i+1}-t) + |\pi_{\hat{\discrete}}| (\hat{t}_{j+1}-t) } .
\end{align*}
If $\hat t_j < t < t_i$, then we can use \Cref{kappa-by-mass-preservation} and \Cref{kappa-computed} to get
\begin{align*}
	\kappa_{i+1,j+1}(t) &= \bigg(1- \frac{\hat{h}_j}{h_i}\bigg) \kappa_{i+1,j}(t) + \frac{\hat{h}_j}{h_i } \kappa_{i,j}(t) + \frac{1}{h_i} h_i \klammern{(t-\hat t_j) \land \hat h_j} \\
	&\le \bigg(1- \frac{\hat{h}_j}{h_i}\bigg) (t_{i+1}-t) + \frac{\hat{h}_j}{h_i } (t_i-t) + t-\hat t_j \\
	&= t_{i+1} - \hat t_{j+1} \\
        &= \sqrt{(t_{i+1} - \hat t_{j+1})^2}
	% \\&\le \sqrt{ (t_{i+1} - \hat{t}_{j+1})^2 + |\pi_\discrete| (t_{i+1}-t) + |\pi_{\hat{\discrete}}| (\hat{t}_{j+1}-t) }
 .
\end{align*}
If $t_i \lor \hat t_j \le t < t_{i+1} \land \hat t_{j+1}$, then we can use \Cref{kappa-by-mass-preservation}, \Cref{kappa-computed} and \Cref{abc-lemma} to get
\begin{align*}
	\kappa_{i+1,j+1}(t) &= \bigg(1-\frac{\hat{h}_j}{h_i}\bigg) \kappa_{i+1, j}(t) + \frac{\hat{h}_j}{h_i} \kappa_{i, j}(t) + \frac{1}{h_i} \klammern{(t-t_i) (\hat t_{j+1}-t) + (t_{i+1}-t) (t-\hat t_j) } \\
	&\le \bigg(1- \frac{\hat{h}_j}{h_i}\bigg) (t_{i+1}-t) + \frac{1}{h_i} \klammern{(t-t_{i+1}+h_i) (\hat t_{j+1}-t) + (t_{i+1}-t) (t-\hat t_{j+1}+\hat h_j) } \\
	&= t_{i+1} - t + \hat t_{j+1} - t - \frac{2}{h_i} (t_{i+1} - t) (\hat t_{j+1}-t) \\
	&\le \sqrt{(t_{i+1}-\hat t_{j+1})^2 + h_i (t_{i+1}-t)} \\
	&\le \sqrt{ (t_{i+1} - \hat{t}_{j+1})^2 + |\pi_\discrete| (t_{i+1}-t) } .
\end{align*}
If $t_{i+1} \lor \hat t_j \le t$, then we can use \Cref{kappa-by-mass-preservation} and \Cref{kappa-computed} to get
\begin{align*}
	\kappa_{i+1,j+1}(t) &= \bigg(1-\frac{\hat{h}_j}{h_i}\bigg) \kappa_{i+1, j}(t) + \frac{\hat{h}_j}{h_i} \kappa_{i, j}(t) + \frac{1}{h_i} \klammern{ h_i (\hat t_{j+1} - t)^+ + 0 } \\
	&= (\hat t_{j+1} - t)^+ = \sqrt{\klammern{(\hat t_{j+1} - t)^+}^2} \le \sqrt{|\pi_{\hat\discrete}| (\hat t_{j+1} - t)^+} .
\end{align*}
We now have shown
\[ \kappa_{i+1,j+1} (t) \le \sqrt{ (t_{i+1} - \hat{t}_{j+1})^2 + |\pi_\discrete| (t_{i+1}-t)^+ + |\pi_{\hat{\discrete}}| (\hat{t}_{j+1}-t)^+ } \]
for every $t \ge 0$. The case $\hat h_j > h_i$ can be treated analogously. This concludes the induction.

If $t < 0$, then note that by \Cref{rho-support}, $\rho^{i,j} (\tau, \hat\tau) = 0$ for $\tau < 0$ and $\hat\tau < 0$ and therefore
\begin{align*}
    \kappa_{i, j} (t) &\le \kappa_{i, j} (0)
    \le \sqrt{(t_i - \hat t_j)^2 + |\pi_\discrete| t_i + |\pi_{\hat\discrete}| \hat t_j} .
\end{align*}
This concludes the proof for all $t \in \R$.
\end{proof}

\begin{lemma} \label{sqrt-with-bv}
Let $g \in BV(-1, T; X)$. Then for $i \in \set{0, \ldots, N}$ and $j \in \setsmall{0, \ldots, \hat N}$ we have
\[ \int_{-1}^T \int_{-1}^T \rho^{i, j} (\tau, \hat\tau) \Xnorm{g(\tau) - g(\hat\tau)} \diff\hat\tau \diff\tau \le \sqrt{(t_i-\hat t_j)^2 + |\pi_\discrete| t_i + |\pi_{\hat\discrete}| \hat t_j} \cdot \essVar(g) . \]
\end{lemma}

\begin{proof}
Let $g_r \colon [-1, T] \to X$ be a representative of $g$ and let $\pi \colon -1 = \tilde t_0 < \tilde t_1 < \ldots < \tilde t_{\tilde N} = T$ be a partition of $[-1, T]$. We define $g_\pi\colon [-1, T] \to X$ by
\[ g_\pi \coloneqq g_r(0) + \sum_{k=1}^{n} \klammern{g_r(\tilde t_{k}) - g_r(\tilde t_{k-1}) } \mathbf{1}_{[\tilde t_k, T]} . \]
Then
\begin{align*}
    & g_\pi (\tau) - g_\pi (\hat\tau) = \sum_{k=1}^{n} \klammern{g_r(\tilde t_{k}) - g_r(\tilde t_{k-1})} \klammern{ \mathbf{1}_{[\tilde t_k,T]\times[-1,\tilde t_k)} (\tau, \hat\tau) - \mathbf{1}_{[-1,\tilde t_k)\times[\tilde t_k,T]} (\tau, \hat\tau) }
\end{align*}
for all $\tau, \hat\tau \in [-1,T]$. Therefore, by using \Cref{mass-concentration} and the notation $\kappa_{i,j}$ as defined in \Cref{kappa-definition}, we get
\begin{align*}
	&\phantom{{}={}} \int_{-1}^T \int_{-1}^T \rho^{i, j} (\tau, \hat\tau) \Xnorm{g_\pi(\tau) - g_\pi(\hat\tau)} \diff\tau \diff\hat\tau \\
	&\le \sum_{k=1}^n \Xnorm{ g_r(\tilde t_{k}) - g_r(\tilde t_{k-1}) } \kappa_{i, j} (\tilde t_k) \\
	&\le \sqrt{(t_i - \hat t_j)^2 + |\pi_\discrete| t_i + |\pi_{\hat\discrete}| \hat t_j} \cdot \sum_{k=1}^n \Xnorm{ g_r(\tilde t_{k}) - g_r(\tilde t_{k-1}) } \\
	&\le \sqrt{(t_i - \hat t_j)^2 + |\pi_\discrete| t_i + |\pi_{\hat\discrete}| \hat t_j} \cdot \Var({g_r}) . \numberthis \label{double-integral-step-1}
\end{align*}
Note also that by \Cref{mass-preservation} and \Cref{mass-preservation-remark} we have
\[ \int_{-1}^T \rho^{i, j} (\tau, \hat\tau) \diff\hat\tau \le \hat t_j \lor 1 \]
for all $\tau \in \R$. Therefore
\begin{align*}
    &\phantom{{}={}} \int_{-1}^T \int_{-1}^T \rho^{i, j} (\tau, \hat\tau) \Xnorm{g(\tau) - g_\pi(\tau)} \diff\hat\tau \diff\tau \\
    &= \int_{-1}^T \int_{-1}^T \rho^{i, j} (\tau, \hat\tau) \diff\hat\tau\Xnorm{g_r(\tau) - g_\pi(\tau)} \diff\tau \\
    &\le (\hat t_j \lor 1) \int_{-1}^{T} \Xnorm{g_r(\tau) - g_\pi(\tau)} \diff\tau \\
    &= (\hat t_j \lor 1) \sum_{k=1}^n \int_{\tilde t_{k-1}}^{\tilde t_{k}} \Xnorm{g_r(\tau)-g_\pi(\tau)} \diff\tau \\
    &= (\hat t_j \lor 1) \sum_{k=1}^n \int_{\tilde t_{k-1}}^{\tilde t_{k}} \Xnorm{g_r(\tau)-g_r(\tilde t_k)} \diff\tau \\
    &\le (\hat t_j \lor 1) \sum_{k=1}^n \int_{\tilde t_{k-1}}^{\tilde t_{k}} \Var({g_r|_{[\tilde t_{k-1}, \tilde t_k]}}) \diff\tau \\
    &\le |\pi| (\hat t_j \lor 1) \sum_{k=1}^n \Var({g_r|_{[\tilde t_{k-1}, \tilde t_k]}}) \\
    &= |\pi| (\hat t_j \lor 1) \Var({g_r}). \numberthis \label{double-integral-step-2}
\end{align*}
and analogously
\[ \int_{-1}^T \int_{-1}^T \rho^{i, j} (\tau, \hat\tau) \Xnorm{g(\hat\tau) - g_\pi(\hat\tau)} \diff\hat\tau \diff\tau \le |\pi| (t_i \lor 1) \Var({g_r}) \numberthis \label{double-integral-step-3} . \]
Combining \Cref{double-integral-step-1}, \Cref{double-integral-step-2} and \Cref{double-integral-step-3} we get
\begin{align*}
	&\phantom{{}={}}\int_{-1}^T \int_{-1}^T \rho^{i, j} (\tau, \hat\tau) \Xnorm{g(\tau) - g(\hat\tau)} \diff\hat\tau \diff\tau \\
	&\le \int_{-1}^T \int_{-1}^T \rho^{i, j} (\tau, \hat\tau) \klammern{ \Xnorm{g(\tau) - g_\pi(\tau)} + \Xnorm{g_\pi(\tau) - g_\pi(\hat\tau)} + \Xnorm{g_\pi(\hat\tau) - g(\hat\tau)} } \diff\hat\tau \diff\tau \\
	&\le \sqrt{(t_i - \hat t_j)^2 + |\pi_\discrete| t_i + |\pi_{\hat\discrete}| \hat t_j} \cdot \Var({g_r}) + (t_i \lor 1) |\pi| \Var({g_r}) + (\hat t_j \lor 1) |\pi| \Var({g_r}) .
\end{align*}
Taking the infimum over all partitions $\pi$ of $[-1, T]$ and over all representatives $g_r$ of $g$ completes the proof.
\end{proof}

\subsection{An explicit upper bound} \label{section-2-3-an-explicit-upper-bound}

We are now able to state an upper bound for $\xnorm{u_\discrete (t_i) - u_{\hat\discrete} (\hat t_j)}$, which only depends on the mesh sizes $|\pi_\discrete|$, $|\pi_{\hat\discrete}|$, the given data of the discretizations $u_\discrete^0$, $u_{\hat\discrete}^0$, $f_\discrete$ and $f_{\hat\discrete}$ as well as some arbitrary $(u, v) \in A$ and $g \in BV(0, T; X)$.

\begin{theorem}\label{main-result}
Let $A \subseteq X \times X$ be accretive of type $\omega \in \R$. Let $\discrete, \hat\discrete$ be discretizations, let $u_\discrete \in \C$ be a solution of the implicit Euler scheme \Cref{Euler-scheme} and let $u_{\hat\discrete}$ be a solution of $(E_{\hat\discrete})$.
If $(|\pi_\discrete| \lor |\pi_{\hat\discrete}|) \omega < 1$, then for all $(u, v) \in A$, $g \in BV(0, T; X)$ and all $i \in \set{0, \ldots, N}$ and $j \in \setsmall{0, \ldots, \hat N}$,
\begin{align*}
	\xnorm{u_\discrete (t_i) - u_{\hat\discrete} (\hat t_j)} \le {} &
	 \mathopen{\exp} \left({\varphi((|\pi_\discrete| \lor |\pi_{\hat\discrete}|) \omega) (t_i+\hat t_j) \omega^+}\right) \Bigg( \xnorm{u_\discrete^0 - u} + \xnorm{u_{\hat\discrete}^0 - u} \\*
        & {} + \int_0^{t_i} \Xnorm{f_\discrete (\tau) - g(\tau)} \diff\tau + \int_0^{\hat t_j} \Xnorm{f_{\hat\discrete} (\hat\tau) - g(\hat\tau)} \diff\hat\tau \\*
        & {} + \sqrt{(t_i-\hat t_j)^2 + |\pi_\discrete| t_i + |\pi_{\hat\discrete}| \hat t_j} \cdot \left(\essVar(g) + \Xnorm{g(0+) - v} \right) \Bigg) .
\end{align*}
\end{theorem}

\begin{proof}
We first apply \Cref{implicit-upper-bound} with $\tilde g \in BV(-1, T; X)$ defined by
\[
	\tilde g(\tau) = \begin{cases}
		v &\text{if } \tau < 0, \\
		g(\tau) &\text{if } \tau \ge 0,
	\end{cases}
\]
which gives us
\begin{align*}
    a_{i, j} \le {} & \mathopen{\exp} \left({\varphi((|\pi_\discrete| \lor |\pi_{\hat\discrete}|) \omega) (t_i+\hat t_j) \omega^+}\right) \Bigg( \xnorm{u_\discrete^0 - u} + \xnorm{u_{\hat\discrete}^0 - u} \\*
    & {} + \int_0^{t_i} \Xnorm{f_\discrete (\tau) - g(\tau)} \diff\tau + \int_0^{\hat t_j} \Xnorm{f_{\hat\discrete} (\hat\tau) - g(\hat\tau)} \diff\hat\tau \\*
    &+ \int_{-1}^T \int_{-1}^T \rho^{i,j} (\tau, \hat\tau) \Xnorm{\tilde g(\tau) - \tilde g(\hat \tau)} \diff\hat\tau \diff\tau \bigg) .
\end{align*}
By \Cref{sqrt-with-bv}, we have
\[ \int_{-1}^T \int_{-1}^T \rho^{i,j} (\tau, \hat\tau) \Xnorm{\tilde g(\tau) - \tilde g(\hat \tau)} \diff\hat\tau \diff\tau \le \sqrt{(t_i-\hat t_j)^2 + |\pi_\discrete| t_i + |\pi_{\hat\discrete}| \hat t_j} \cdot \essVar(\tilde g) . \]
Since $\essVar (\tilde g) = \essVar (g) + \Xnorm{g(0+) - v}$, this completes the proof.
\end{proof}

\begin{remark}[Properties of $\varphi$]
	Recall that $\varphi \colon (-\infty, 1) \to \R$ is given by
	\[ \varphi(x) = \begin{cases} \ds \frac{-\log(1-x)}{x} &\text{if } x \neq 0, \\ 1 &\text{if } x = 0. \end{cases} \]
	Having the limit
	\[ \lim_{x \to 0} \varphi (x) = 1 \]
	is quite advantageous when the mesh size tends to 0.
	Note that for $-1 \le x < 1$ we have the series representation
	\[ \varphi(x) = \sum_{k=0}^\infty \frac{x^k}{k+1} = 1 +\frac{x}{2} + \frac{x^2}{3} + \frac{x^3}{4} + \cdots . \]
	By using $\varphi(x) \le 2$ for $0 \le x \le \frac{1}{2}$, we can see, that \Cref{main-result} is a generalization of the result by Kobayashi \cite[Lemma 2.1]{Kobayashi75}.
\end{remark}

\begin{theorem}[Kobayashi 1975]
	If $(|\pi_\discrete| \lor |\pi_{\hat\discrete}|) \omega \le \frac 1 2$, then for all $(u, v) \in A$ and all $i \in \set{0, \ldots, N}$ and $j \in \setsmall{0, \ldots, \hat N}$, we have
	\begin{align*}
		\xnorm{u_\discrete (t_i) - u_{\hat\discrete} (\hat t_j)} \le {} & \mathopen{\exp}  \klammern{2(t_i+\hat t_j) \omega^+} \Big(  \xnorm{u_\discrete^0 - u} + \xnorm{u_{\hat\discrete}^0 - u} + \int_0^{T} \Xnorm{f_\discrete (\tau)} \diff\tau \\*
		&{}+ \int_0^{T} \Xnorm{f_{\hat\discrete} (\hat\tau)} \diff\hat\tau + \sqrt{(t_i-\hat t_j)^2 + |\pi_\discrete| t_i + |\pi_{\hat\discrete}| \hat t_j} \cdot \Xnorm{v} \Big) .
	\end{align*}
\end{theorem}

\begin{proof}
	Apply \Cref{main-result} with $g = 0$ and use $\varphi(x) \le 2$ for $0 \le x \le \frac{1}{2}$.
\end{proof}

\begin{remark}
    Note that \Cref{main-result} only evaluates $u_\discrete$ and $u_{\hat\discrete}$ at time points in the partitions $\pi_\discrete$ and $\pi_{\hat\discrete}$. Therefore, \Cref{main-result} is also true, if the solutions to Euler schemes were defined to be piecewise constant.
    The following version of our main result however uses our definition of piecewise affine solutions of Euler schemes to give an upper bound for $\xnorm{u_\discrete (t) - u_{\hat\discrete} (\hat t)}$ for all $t, \hat t \in [0, T]$.
\end{remark}

\begin{theorem}\label{main-result-continuous}
Let $A \subseteq X \times X$ be accretive of type $\omega \in \R$. Let $\discrete, \hat\discrete$ be discretizations, let $u_\discrete \in \C$ be a solution of the implicit Euler scheme \Cref{Euler-scheme} and let $u_{\hat\discrete}$ be a solution of $(E_{\hat\discrete})$.
If $(|\pi_\discrete| \lor |\pi_{\hat\discrete}|) \omega < 1$, then for all $(u, v) \in A$, $g \in BV(0, T; X)$ and all $t, \hat t \in [0, T]$ we have
\begin{align*}
	&\xnorm{u_\discrete (t) - u_{\hat\discrete} (\hat t)} \\ 
	\le {}
    & \mathopen{\exp} \left({\varphi((|\pi_\discrete| \lor |\pi_{\hat\discrete}|) \omega) (\ceiling{t}_{\pi_\discrete}+\ceiling{\hat t}_{\pi_{\hat\discrete}}) \omega^+}\right)
    \Bigg( \xnorm{u_\discrete^0 - u} + \xnorm{u_{\hat\discrete}^0 - u} \\
    & + \int_0^{\ceiling{t}_{\pi_\discrete}} \Xnorm{f_\discrete (\tau) - g(\tau)} \diff\tau + \int_0^{\ceiling{\hat t}_{\pi_{\hat\discrete}}} \Xnorm{f_{\hat\discrete} (\hat\tau) - g(\hat\tau)} \diff\hat\tau \\
	& {} + \sqrt{(|t-\hat t| + |\pi_\discrete| + |\pi_{\hat\discrete}|)^2 + |\pi_\discrete| t + |\pi_{\hat\discrete}| \hat t} \cdot \left(\essVar(g) + \Xnorm{g(0+) - v} \right) \Bigg) .
\end{align*}
\end{theorem}

\begin{proof}
    Choose $i \in \setsmall{0, \ldots, N-1}$ and $j \in \setsmall{0, \ldots, \hat N-1}$, such that $t \in [t_i, t_{i+1}]$ and $\hat t \in [\hat t_j, \hat t_{j+1}]$, so there exist $\lambda, \hat\lambda \in [0, 1]$ such that $t = \lambda t_i + (1-\lambda) t_{i+1}$ and $\hat t = \hat\lambda \hat t_j + (1-\hat\lambda) \hat t_{j+1}$. Since $u_\discrete$ and $u_{\hat\discrete}$ are piecewise affine, we get
\begin{align*}
	\xnorm{u_\discrete (t) - u_{\hat\discrete} (\hat t)} &= \xnorm{\lambda u_\discrete (t_i) + (1-\lambda) u_\discrete (t_{i+1}) - \hat\lambda u_{\hat\discrete} (\hat t_j) - (1-\hat\lambda) u_{\hat\discrete} (\hat t_{j+1})} \\
	&\le \lambda \klammern{ \hat\lambda a_{i,j} + (1-\hat\lambda)a_{i,j+1} } + (1-\lambda) \klammern{ \hat\lambda a_{i+1,j} + (1-\hat\lambda) a_{i+1,j+1} }.
\end{align*}
For $k \in \set{i, i+1}$ and $\ell \in \set{j, j+1}$, we can use \Cref{main-result} to get an upper bound for $a_{k, \ell}$.
Note that $t_k \le \ceiling{t}_{\pi_\discrete}$ and $\hat t_\ell \le \ceiling{\hat t}_{\pi_{\hat\discrete}}$. By using $|t_k - \hat t_\ell| \le |t-\hat t| + |\pi_\discrete| + |\pi_{\hat\discrete}|$ and the concavity of the square root function, we get
\begin{align*}
    &\lambda \hat\lambda \sqrt{(t_i-\hat t_j)^2 + |\pi_\discrete| t_i + |\pi_{\hat\discrete}| \hat t_j} 
    + \lambda (1 - \hat\lambda) \sqrt{(t_i-\hat t_{j+1})^2 + |\pi_\discrete| t_i + |\pi_{\hat\discrete}| \hat t_{j+1}} \\*
    &+ (1 - \lambda) \hat\lambda \sqrt{(t_{i+1}-\hat t_j)^2 + |\pi_\discrete| t_{i+1} + |\pi_{\hat\discrete}| \hat t_j} \\*
    &+ (1 - \lambda) (1 - \hat\lambda) \sqrt{(t_{i+1}-\hat t_{j+1})^2 + |\pi_\discrete| t_{i+1} + |\pi_{\hat\discrete}| \hat t_{j+1}} \\
    {}\le{} & \Big( (|t-\hat t| + |\pi_\discrete| + |\pi_{\hat\discrete}|)^2 + \lambda \hat\lambda (|\pi_\discrete| t_i + |\pi_{\hat\discrete}| \hat t_j)
    + \lambda (1 - \hat\lambda) (|\pi_\discrete| t_i + |\pi_{\hat\discrete}| \hat t_{j+1}) \\*
    &+ (1 - \lambda) \hat\lambda ( |\pi_\discrete| t_{i+1} + |\pi_{\hat\discrete}| \hat t_j ) 
    + (1 - \lambda) (1 - \hat\lambda) ( |\pi_\discrete| t_{i+1} + |\pi_{\hat\discrete}| \hat t_{j+1} ) \Big)^{\frac{1}{2}} \\
    {}={} &\sqrt{(|t-\hat t| + |\pi_\discrete| + |\pi_{\hat\discrete}|)^2 + |\pi_\discrete| t + |\pi_{\hat\discrete}| \hat t} .
\end{align*}
This completes the proof.
\end{proof}

\begin{corollary}[Distance in $\C$] \label{distance-in-C}
	If $(|\pi_\discrete| \lor |\pi_{\hat\discrete}|) \omega < 1$, then for all $(u, v) \in A$ and $g \in BV(0, T; X)$ we have
	\begin{align*}
		\norm{u_{\discrete} - u_{\hat\discrete}}_{\C} \le {} & \mathopen{\exp}\klammern{\varphi((|\pi_\discrete| \lor |\pi_{\hat\discrete}|) \omega) 2 T \omega^+}
		\Big( \xnorm{u_\discrete^0 - u} + \xnorm{u_{\hat\discrete}^0 - u} \\
		& + \norm{f_\discrete-g}_{L^1(0, T; X)} + \norm{f_{\hat\discrete}-g}_{L^1(0, T; X)} \\
		& {} + \sqrt{(|\pi_\discrete| + |\pi_{\hat\discrete}|)^2 + |\pi_\discrete| T + |\pi_{\hat\discrete}| T} \cdot \left(\essVar(g) + \Xnorm{g(0+) - v} \right) \Big) .
	\end{align*}
\end{corollary}

\begin{proof}
	Use \Cref{main-result-continuous} and choose $t = \hat t$. Using $t \le \ceiling{t}_{\pi_\discrete} \le T$ as well as $\hat t \le \ceiling{\hat t}_{\pi_{\hat\discrete}} \le T$ yields the desired result.
\end{proof}

\section{Aplications}

In this section we show how our results can be used to establish existence and uniqueness of Euler solutions for the Cauchy problem \Cref{Cauchy-Problem} as well as showing some properties of the Euler solution.

\begin{theorem} \label{wellposedness-for-accretive-and-existing-Euler-sequence}
	Let $A \subseteq X \times X$ be quasi-accretive, $f \in L^1(0, T; X)$ and $u^0 \in \overline{\dom A}$. If there exists a sequence of discretizations $(\discrete_n)$ of the form $\discrete_n = (\pi_{\discrete_n}, f_{\discrete_n}, u^0_{\discrete_n})$ satisfying
	\begin{align*}
		&\lim_{n\to\infty} |\pi_{\discrete_n}| = 0, \\
		&\lim_{n\to\infty} \norm{f_{\discrete_n} - f}_{L^1(0, T; X)} = 0, \\
		&\lim_{n\to\infty} \xnorm{u^0_{\discrete_n} - u^0} = 0
	\end{align*}
	and if the implicit Euler scheme $(E_{\discrete_n})$ has a solution for every $n \in \N$, then there exists a unique Euler solution of the Cauchy problem \Cref{Cauchy-Problem}.
\end{theorem}

\begin{proof}
	Let $A$ be accretive of type $\omega \in \R$. Let the sequence $(\discrete_n)$ be as in the assumption. For every $n \in \N$ let $u_n \in \C$ be a solution of the implicit Euler scheme $(E_{\discrete_n})$, which exists by assumption.
	Since $\lim_{n\to\infty} |\pi_{\discrete_n}| = 0$, there exists an $n_0 \in \N$ such that $|\pi_{\discrete_n}| \omega < 1$ for all $n \in \N_{\ge n_0}$.
	By \Cref{distance-in-C}, for every $n, m \in \N_{\ge n_0}$, every $(\hat u, \hat v) \in A$ and every $g \in BV(0, T; X)$, we get
	\begin{align*}
		& \norm{u_n - u_m}_\C \\* 
		\le {} &
		\mathopen{\exp}{\klammern{\varphi((|\pi_{\discrete_n}| \lor |\pi_{\discrete_m}|) \omega) 2 T \omega^+}}
		\Big( \xnorm{u_{\discrete_n}^0 - \hat u} + \xnorm{u_{\discrete_m}^0 - \hat u} \\*
		& + \norm{f_{\discrete_n}-g}_{L^1(0, T; X)} + \norm{f_{\discrete_m}-g}_{L^1(0, T; X)} \\
		& {} + \sqrt{(|\pi_{\discrete_n}| + |\pi_{\discrete_m}|)^2 + |\pi_{\discrete_n}| T + |\pi_{\discrete_m}| T} \cdot \left(\essVar(g) + \Xnorm{g(0+) - \hat v} \right) \Big) .
	\end{align*}
	As a consequence,
	\[ \limsup_{n, m \to \infty} \norm{u_n - u_m}_{C([0, T]; X)} \le 2 e^{2T\omega^+} \klammern{ \xnorm{u^0-\hat u} + \norm{f - g}_{L^1(0, T; X)} } . \]
	Since $\dom A$ is dense in its closure and $BV(0, T; X)$ is dense in $L^1(0, T; X)$, the right hand side can be made arbitrarily small, and it follows that $(u_n)$ is a Cauchy sequence and hence convergent in $C([0,T];X)$. By definition, the limit $u$ is an Euler solution of the Cauchy problem \Cref{Cauchy-Problem}, and we have proved existence.
	
	Now let $u, \hat u \in \C$ be two Euler solutions of \Cref{Cauchy-Problem}. Then there are sequences of discretizations $(\discrete_n)_n = ((\pi_{\discrete_n}, f_{\discrete_n}, u^0_{\discrete_n}))_n$ and $(\hat\discrete_n)_n = ((\pi_{\hat\discrete_n}, f_{\hat\discrete_n}, u^0_{\hat\discrete_n}))_n$ with
	\begin{align*}
		&\lim_{n\to\infty} |\pi_{\discrete_n}| = \lim_{n\to\infty} |\pi_{\hat\discrete_n}| = 0 , \\
		&\lim_{n\to\infty} \norm{f_{\discrete_n} - f}_{L^1(0, T; X)} = \lim_{n\to\infty} \normsmall{f_{\hat\discrete_n} - f}_{L^1(0, T; X)} = 0, \\
		&\lim_{n\to\infty} \xnorm{u^0_{\discrete_n} - u^0} = \lim_{n\to\infty} \xnorm{u^0_{\hat\discrete_n} - u^0} = 0
	\end{align*}
	and there are sequences $(u_n)_n$ and $(\hat u_n)_n$ in $\C$, such that $u_n$ is a solution of $(E_{\discrete_n})$ and $\hat u_n$ is a solution of $(E_{\hat\discrete_n})$ for all $n \in \N$ and
	\[ \lim_{n \to \infty} \norm{u_n - u}_\C = \lim_{n \to \infty} \norm{\hat u_n - \hat u}_\C  = 0 . \]
	Since $\lim_{n\to\infty} |\pi_{\discrete_n}| = \lim_{n\to\infty} |\pi_{\hat\discrete_n}| = 0$, there is an $n_0 \in \N$ such that $(|\pi_{\discrete_n}| \lor |\pi_{\hat\discrete_n}|) \omega < 1$ for all $n \in \N_{\ge n_0}$. By \Cref{distance-in-C}, for every $n \in \N_{\ge n_0}$, every $(\tilde u, \tilde v) \in A$ and every $g \in BV(0, T; X)$ we get
	\begin{align*}
		& \norm{u_n - \hat u_n}_\C \\* 
		\le {} &
		\mathopen{\exp}{\klammern{\varphi((|\pi_{\discrete_n}| \lor |\pi_{\hat\discrete_n}|) \omega) 2 T \omega^+}}
		\Big( \xnorm{u_{\discrete_n}^0 - \tilde u} + \xnorm{u_{\hat\discrete_n}^0 - \tilde u} \\*
		& + \norm{f_{\discrete_n}-g}_{L^1(0, T; X)} + \normsmall{f_{\hat\discrete_n}-g}_{L^1(0, T; X)} \\
		& {} + \sqrt{(|\pi_{\discrete_n}| + |\pi_{\hat\discrete_n}|)^2 + |\pi_{\discrete_n}| T + |\pi_{\hat\discrete_n}| T} \cdot \left(\essVar(g) + \Xnorm{g(0+) - \tilde v} \right) \Big) .
	\end{align*}
	Therefore
	\begin{align*}
		\norm{u - \hat u}_\C &= \lim_{n \to \infty} \norm{u_n - \hat u_n}_\C \le \mathopen{2 e^{2 T \omega^+}} \klammern{ \xnorm{u^0 - \tilde u} + \norm{f - g}_{L^1(0, T; X)} } .
	\end{align*}
	Since $\dom A$ is dense in its closure and $BV(0, T; X)$ is dense in $L^1(0, T; X)$, the right hand side can be made arbitrarily small. Thus, $u = \hat u$ and we have proved uniqueness.
\end{proof}

We are now able to establish the wellposedness of \Cref{Cauchy-Problem} as shown by Crandall and Evans in \cite{CrandallEvans1975} as well as the estimates \Cref{wellposedness-1} and \Cref{wellposedness-2}.

\begin{theorem}
\label{wellposedness}
    Let $A \subseteq X \times X$ be $m$-accretive of type $\omega \in \R$. Then for every $f \in L^1(0, T; X)$ and for every $u^0 \in \overline{\dom A}$ there exists a unique Euler solution $u \in C([0, T]; X)$ of \Cref{Cauchy-Problem}.
    
    If for $(\hat u, \hat v) \in A$ we define $f_{\hat v} \in L^1(-T, T; X)$ by
    \[ f_{\hat v}(\tau) \coloneqq \begin{cases}
        f(\tau) &\text{if } \tau \ge 0 , \\
        \hat v &\text{if } \tau < 0 ,
    \end{cases} \]
    then, for all $t, \hat t \in [0, T]$,
    \[ \xnorm{u(t) - u(\hat t)} \le \left( e^{t \omega} + e^{\hat t \omega} \right) \xnorm{u^0-\hat u} + \int_0^{t \lor \hat t} e^{\tau \omega} \xnorm{f_{\hat v} (t-\tau) - f_{\hat v} (\hat t - \tau)} \diff\tau \numberthis \label{wellposedness-1} . \]
    Moreover, if $f, \widehat f \in L^1(0, T; X)$ and $u^0, \hat u^0 \in \overline{\dom A}$, $u$ is the Euler solution of \Cref{Cauchy-Problem} and $\hat u$ is the Euler solution of $\textnormal{(CP}\colon \widehat f, \hat u^0\textnormal{)}$,
    then, for every $t \in [0, T]$,
    \begin{align*}
        \xnorm{ u (t) - \hat u (t) } &\le e^{t \omega} \xnorm{u^0-\hat u^0} + \int_0^t e^{(t-s)\omega} [u(s)-\hat u(s), f(s) - \widehat f(s)] \diff s . \numberthis \label{wellposedness-2}
    \end{align*}
\end{theorem}

\begin{proof}
    Let $(\discrete_n)$ be a sequence of discretizations of the form $\discrete_n = (\pi_{\discrete_n}, f_{\discrete_n}, u^0_{\discrete_n})$ satisfying
    \begin{align*}
        &\sup_{n \in \N} |\pi_{\discrete_n}| \omega < 1, \\
        &\lim_{n\to\infty} |\pi_{\discrete_n}| = 0, \\
        &\lim_{n\to\infty} \norm{f_{\discrete_n} - f}_{L^1(0, T; X)} = 0 \quad \text{ and} \\
        &\lim_{n\to\infty} \xnorm{u^0_{\discrete_n} - u^0} = 0.
    \end{align*}
    Note that such a sequence exists for every $f \in L^1(0, T; X)$ since step functions are dense in $L^1(0, T; X)$. Specifically, one could choose the conditional expectation
    \[ f_{\discrete_n} (t) \coloneqq \frac{1}{t_{i+1} - t_i} \int_{t_i}^{t_{i+1}} f(\tau) \diff\tau \]
    for $t \in [t_i, t_{i+1})$, $i \in \set{0, \ldots, N-1}$. 
    Since $A$ is $m$-accretive and $|\pi_{\discrete_n}| \omega < 1$, the Euler scheme $(E_{\discrete_n})$ has a solution for all $n \in \N$, so by \Cref{wellposedness-for-accretive-and-existing-Euler-sequence}, the Cauchy problem \Cref{Cauchy-Problem} has a unique Euler solution.

    To prove $\Cref{wellposedness-1}$, we take a sequence of discretizations $(\hat \discrete _n)$ of the form 
    \[ \hat \discrete_n = (\pi_n \colon 0 = \frac{0}{n} T < \frac{1}{n} T < \ldots < \frac{n}{n} T = T , f_n, u^0) , \]
    such that
    \[
        \lim_{n\to\infty} \norm{f_n - f}_{L^1(0, T; X)} = 0 .
    \]
    For every $n \in \N$ let $u_n \in C([0, T]; X)$ be the solution of the corresponding Euler scheme $(E_{\hat\discrete_n})$.
    By \Cref{equidistant-time}, for all $n \in \N_{> \omega T}$, every $(\hat u, \hat v) \in A$ and all $t^n_i, t^n_j \in \pi_n$
    \begin{align*}
        \xnorm{u_n(t^n_i) - u_n(t^n_j)} \le{}& \mathopen{\exp}\klammern{{\varphi(|\pi_n|\omega)) t^n_i \omega}} \xnorm{u^0-\hat u} + \mathopen{\exp}\klammern{{\varphi(|\pi_n|\omega)) t^n_j \omega}} \xnorm{u^0-\hat u} \\*
        &+ \int_0^{|t^n_i-t^n_j|} \exp( \varphi(|\pi_n|\omega)) ((t^n_i \lor t^n_j) - \floor{\tau}_{\pi_n}) \omega ) \xnorm{f_n (\tau) - \hat v} \diff\tau \\*
        &+ \int_0^{t^n_i \land t^n_j} \exp(\varphi(|\pi_n|\omega)) ((t^n_i \land t^n_j) - \floor{\tau}_{\pi_n}) \omega) \\*
        &\hphantom{{}+\int_0^{t^n_i \land t^n_j}} \xnorm{ f_n(\tau + (t^n_i-t^n_j)^+) - f_n(\tau + (t^n_j-t^n_i)^+) } \diff\tau .
    \end{align*}
    Taking the limit as $n \to \infty$, $t^n_i \to t$ and $t^n_j \to \hat t$ yields
    \begin{align*}
        \xnorm{u(t) - u(\hat t)} \le {} & \klammern{ e^{t \omega} + e^{\hat t \omega} } \xnorm{u^0-\hat u} + \int_0^{|t-\hat t|} e^{((t \lor \hat t) - \tau) \omega} \xnorm{f(\tau) - \hat v} \diff\tau \\*
        & + \int_0^{t_i \land \hat t_j} e^{((t \land \hat t)-\tau)\omega} \xnorm{f(\tau + (t-\hat t)^+)) - f(\tau + (\hat t-t)^+)} \\
        {}={} & \left( e^{t \omega} + e^{\hat t \omega} \right) \xnorm{u^0-\hat u} + \int_0^{t \lor \hat t} e^{\tau \omega} \xnorm{f_{\hat v} (t-\tau) - f_{\hat v} (\hat t - \tau)} \diff\tau .
    \end{align*}

	To prove $\Cref{wellposedness-2}$, we take two sequences of discretizations $(\discrete _n)$ and $(\hat \discrete _n)$ of the form 
	\begin{align*}
		\discrete_n &= (\pi_n \colon 0 = \frac{0}{n} T < \frac{1}{n} T < \ldots < \frac{n}{n} T = T , f_n, u_n^0) \quad \text{and} \\*
		\hat\discrete_n &= (\pi_n, \widehat f_n, \hat u_n^0)
	\end{align*}
	such that
	\begin{align*}
		&\lim_{n\to\infty} \norm{f_n - f}_{L^1(0, T; X)} = 0, \\
		&\lim_{n\to\infty} \normsmall{\widehat f_n - \widehat f}_{L^1(0, T; X)} = 0, \\
		&\lim_{n\to\infty} \xnorm{u_n^0 - u^0} = 0 \quad \text{and} \\
		&\lim_{n\to\infty} \xnorm{\hat u_n^0 - \hat u^0} = 0.
	\end{align*}
	Let $u_n \in C([0, T]; X)$ be the solution of the corresponding Euler scheme $(E_{\discrete_n})$ and let $\hat u_n \in C([0, T]; X)$ be the solution of $(E_{\hat\discrete_n})$.
	By \Cref{equidistant-time}, for all $n \in \N_{> \omega T}$, every $(\hat u, \hat v) \in A$ and all $t^n_i \in \pi_n$,
	\begin{align*}
		\xnorm{u_n(t^n_i) - \hat u_n(t^n_i)} \le{}& \mathopen{\exp}\klammern{{\varphi(|\pi_n|\omega) t^n_i \omega}} \left( \xnorm{u^0-\hat u} + \xnorm{u^0-\hat u} \right) \\*
		&+ \int_0^{t^n_i} \mathopen{\exp} (\varphi(|\pi_n|\omega) (t^n_i - \floor{\tau}_{\pi_n}) \omega) \xnorm{ f_n(\tau) - \widehat f_n(\tau) } \diff\tau .
	\end{align*}
	Taking the limit as $n \to \infty$ and $t^n_i \to t$ yields
	\[ \xnorm{ u (t) - \hat u (t) } \le e^{t \omega} \xnorm{u^0-\hat u^0} + \int_0^t e^{(t-s)\omega} [u(s)-\hat u(s), f(s) - \widehat f(s)] \diff s . \]
	This completes the proof.
\end{proof}

\begin{theorem}[Crandall-Liggett \cite{CrandallLiggett1971}]
	Let $A \subseteq X \times X$ and $\omega \in \R$.
	Suppose that $A$ is accretive of type $\omega$ and that $A$ satisfies the range condition.
	Then for every initial value $u^0 \in \overline{\dom A}$ the Cauchy problem \CP{0, u^0} admits a unique Euler solution $u \in \C$.
	Moreover, if $u$ is an Euler solution for initial value $u^0 \in \overline{\dom A}$ and $\hat u$ is an Euler solution for initial value $\hat u^0 \in \overline{\dom A}$, then, for every $t \in [0, T]$,
	\[ \xnorm{ u(t) - \hat u(t) } \le e^{\omega t} \xnorm{u^0 - \hat u^0} . \]
	If, in addition, $u^0 \in \dom A$, then $u$ is Lipschitz continuous with Lipschitz constant $e^{\omega T} \setnorm{A u^0}$.
\end{theorem}

\begin{proof}
	The existence and uniqueness of solutions, as well as the stability estimate, are proved in the same way as in the proof of \Cref{wellposedness}. The only difference is the existence of solutions of appropriate implicit Euler schemes. For zero right hand sides and for initial values $u^0 \in \overline{\dom A}$, this existence follows from the range condition.
\end{proof}

\begin{definition} \label{Definition-generalized-domain}
	For any $B \subseteq X$ we define the \define{set norm}
	\[ \setnorm{B} \coloneqq \inf_{x \in B} \xnorm{x} . \]
	Let $A \subseteq X \times X$. For $x \in X$ we define
	\[ \Anorm{x} \coloneqq \liminf_{\hat x \to x} \setnorm{A \hat x} = \sup_{r > 0} \inf_{\hat x \in B(x, r)}  \setnorm{A \hat x} , \]
	where $B(x, r) \coloneqq \set{\hat x \in X \colon \xnorm{x - \hat x} \le r}$.
	The set
	\[ \generalizeddom A \coloneqq \set{x \in X \colon \Anorm{x} < \infty } \]
	is called the \define{generalized domain of $A$}.
\end{definition}

\begin{remark}
	Let $A \subseteq X \times X$. Since $\Anorm{x} = \liminf_{\hat x \to x} \setnorm{A \hat x} \le |Ax|$ for all $x \in X$ and $\Anorm{x} = \infty$ if $x \not \in \overline{\dom A}$, we have
	\[ \dom A \subseteq \generalizeddom A \subseteq \overline{\dom A} . \]
\end{remark}

\begin{lemma}
	Let $A \subseteq X \times X$ and $u, x \in X$. Then $ |u|_{A+x} \le \Anorm{u} + \xnorm{x} . $
\end{lemma}

\begin{proof}
	As a direct consequence of \Cref{Definition-generalized-domain}, we get
	\begin{align*}
		|u|_{A+x} &= \liminf_{\hat x \to x} \setnorm{A \hat x + x} \\
		&= \liminf_{\hat x \to x}  \inf_{v \in A \hat x} \xnorm{v + x} \\
		&\le \liminf_{\hat x \to x}  \inf_{v \in A \hat x} \klammern{ \xnorm{v} + \xnorm{x} } \\
		&= \liminf_{\hat x \to x} \setnorm{A \hat x} + \xnorm{x} \\
		&= \Anorm{u} + \xnorm{x} .
	\end{align*}
	This completes the proof.
\end{proof}

\begin{corollary} \label{Lipschitz-if-BV-and-dom-A}
    Let $A \subseteq X \times X$ be $m$-accretive of type $\omega \in \R$, $f \in BV(0, T; X)$, $u^0 \in {\generalizeddom A}$ 
    and $u \in C([0, T]; X)$ be the Euler solution of \Cref{Cauchy-Problem}. Then $u$ is Lipschitz continuous with Lipschitz constant
	\[ \mathopen{ e^{T\omega^+} } \klammern{ |u^0|_{A-f(0+)} + \essVar(f) } . \]
\end{corollary}

\begin{proof}
	Let $(u_n, v_n)$ be a sequence in $A$ with 
	\begin{align*}
		&\lim_{n \to \infty} \xnorm{u_n -u^0} = 0 \quad \text{and} \\
		&\lim_{n \to \infty} \Xnorm{v_n - f(0+)} = \liminf_{\hat u \to u^0} \setnorm{A \hat u - f(0+)} = |u^0|_{A-f(0+)} .
	\end{align*}
	Let $t, \hat t \in [0, T]$. By \Cref{wellposedness}, we can use \Cref{wellposedness-1} and get
	\begin{align*} \xnorm{u(t) - u(\hat t)} &\le \left( e^{t \omega} + e^{\hat t \omega} \right) \xnorm{u^0-u_n} + \int_0^{t \lor \hat t} e^{\tau \omega} \xnorm{f_{v_n} (t-\tau) - f_{v_n} (\hat t - \tau)} \diff\tau \\
		&\le \mathopen{e^{T\omega^+}} \left( 2 \xnorm{u^0-u_n} + \int_0^{t \lor \hat t} \xnorm{f_{v_n} (t-\tau) - f_{v_n} (\hat t - \tau)} \diff\tau \right) \\
		&\le \mathopen{e^{T\omega^+}} \left( 2 \xnorm{u^0-u_n} + \essVar (f_{v_n}) |t-\hat t| \right) \\
		&= \mathopen{e^{T\omega^+}} \left( 2 \xnorm{u^0-u_n} + \left( \xnorm{v_n - f(0+)} + \essVar(f) \right) |t-\hat t| \right)
	\end{align*}
	for all $n \in \N$. Here we used \Cref{BV-shift-estimate} to estimate the integral term. Taking the limit as $n \to \infty$ we get
	\[  \xnorm{u(t) - u(\hat t)} \le \mathopen{e^{T\omega^+}} \klammern{ |u^0|_{A-f(0+)} + \essVar(f) } |t-\hat t| . \]
	This completes the proof.
\end{proof}

\begin{definition}
	A Banach space $X$ has the \define{Radon-Nikodym property} if
	\[ Lip(0, T; X) = W^{1, \infty} (0, T; X) . \]
\end{definition}

\begin{example}
	Every reflexive Banach space (and therefore every Hilbert space) and every separable dual space has the Radon-Nikodym property \cite{DiestelUhl1977}.
\end{example}

\begin{definition}
    A function $u \in W^{1, 1} (0, T; X)$ is a \define{strong solution} of \Cref{Cauchy-Problem}, if $u(0) = u^0$ and
	\[ \dot u (t) + A u(t) \ni f(t) \quad \text{for almost every } t \in [0, T] . \]
 A function $u \in \C$ is a \define{mild solution} of \Cref{Cauchy-Problem}, if there exist sequences $(u_n)$ in $W^{1, 1} (0, T; X)$, $(f_n)$ in $L^1 (0, T; X)$ and $(u^0_n)$ in $X$, such that $u_n$ is a strong solution of \CP{f_n, u^0_n} for all $n \in \N$,
	\begin{align*}
		& \lim_{n\to\infty} \norm{f_n - f}_{L^1(0, T; X)} = 0 , \\
		& \lim_{n\to\infty} \xnorm{u^0_n - u^0} = 0 \text{ and} \\
		& \lim_{n\to\infty} \norm{u_n - u}_\C = 0 .
	\end{align*}
\end{definition}

\begin{corollary} \label{Euler-solutions-are-strong-solutions}
	Let $X$ be a Banach space, which has the Radon-Nikodym property, and let $A \subseteq X \times X$ be quasi-$m$-accretive. Then the following statements are true.
	
	\truestatements{
		\item For every $f \in BV(0, T; X)$ and every $u^0 \in {\generalizeddom A}$ every Euler solution of \Cref{Cauchy-Problem} is a strong solution of \Cref{Cauchy-Problem}.
		\item For every $f \in L^1(0, T; X)$ and every $u^0 \in \overline{\dom A}$ every Euler solution of \Cref{Cauchy-Problem} is a mild solution of \Cref{Cauchy-Problem}.
	}
\end{corollary}

\begin{proof}
	Let $X$ be a Banach space, which has the Radon-Nikodym property, and let $A \subseteq X \times X$ be $m$-accretive of type $\omega \in \R$.
	
	Let $f \in BV(0, T; X)$, $u^0 \in {\generalizeddom A}$ and $u$ be the Euler solution of \Cref{Cauchy-Problem}. Then by \Cref{Lipschitz-if-BV-and-dom-A}, $u \in Lip(0, T; X) = W^{1, \infty} (0, T; X)$. By \cite{Benilan1999}, $u$ is a strong solution of \Cref{Cauchy-Problem}.
	
	Now let $f \in L^1(0, T; X)$, $u^0 \in \overline{\dom A}$ and $u$ be the Euler solution of \Cref{Cauchy-Problem}. Since $BV(0, T; X)$ is dense in $L^1(0, T; X)$ and $\generalizeddom A$ is dense in $\overline{\dom A}$, there exist sequences $(f_n)$ in $BV(0, T; X)$ and $(u^0_n)$ in $\generalizeddom A$, such that
	\begin{align*}
		& \lim_{n\to\infty} \norm{f_n - f}_{L^1(0, T; X)} = 0 \quad \text{and} \\
		& \lim_{n\to\infty} \xnorm{u^0_n - u^0} = 0.
	\end{align*}
	For $n \in \N$ let $u_n$ be the Euler solution of \Cref{Cauchy-Problem}. By \Cref{Euler-solutions-are-strong-solutions} (a), $u_n$ is a strong solution of \Cref{Cauchy-Problem} and by \Cref{wellposedness},
	\[ \norm{u_n - u}_\C \le e^{T \omega^+} \klammern{ \xnorm{u^0_n - u^0} + \norm{f_n - f}_{L^1(0, T; X)} } . \]
	Therefore, \[ \lim_{n\to\infty} \norm{u_n - u}_\C = 0 \]
	and $u$ is a mild solution of \Cref{Cauchy-Problem}.
\end{proof}

\appendix

\section{Functions of bounded variation} \label{appendix}

In this appendix, we show some well known results concerning functions of bounded variation. Let $X$ be a Banach space with norm $\Xnorm{\cdot}$ and $a, b \in \R$ with $a < b$.

\begin{definition}
A function $f \colon [a, b] \to X$ is \define{of bounded pointwise variation}, if
\[ \variation{f} \coloneqq \sup \set{ \sum_{k=0}^{n-1} \Xnorm{f(t_{k+1})-f(t_k)} : n \in \N , t_0, \ldots, t_n \in [a, b], t_0 < \ldots < t_n } < \infty \]
and $f$ is \define{of bounded (essential) variation}, if
\[ \essVar(f) \coloneqq \inf \set{ \Var (g) \mid g\colon [a, b] \to X, f(t) = g(t) \text{ for a.e. } t \in I } < \infty . \]
For $f \in L^1(a, b; X)$ we define $\essVar(f) \coloneqq \essVar(f_r)$, where $f_r \colon [a, b] \to X$ is a representative of $f$. We denote the space of all functions in $L^1(a, b; X)$ with bounded essential variation by $BV(a, b; X)$ and endow it with the norm
\[ \norm{f}_{BV(a, b; X)} \coloneqq \norm{f}_{L^1(a, b; X)} + \essVar(f) . \numberthis \label{BV-norm-definition} \]
\end{definition}

\begin{lemma} \label{BV-continuity}
    If $f \colon [a, b] \to X$ is of bounded pointwise variation, then $f$ has a limit from the right $f(t+)$ at every $t \in [a, b)$ and a limit from the left $f(t-)$ at every $t \in (a, b]$ and 
    \[ f(t+) = f(t-) = f(t) \numberthis \label{bv-a-e-continuous} \]
    for all $t \in (a,b)$ except on a countable set.
\end{lemma}

\begin{proof}
	Let us first assume, that there exists a $t \in [a, b)$, such that the limit from the right $f(t+)$ does not exist. Then there exists a strictly decreasing sequence $(t_n)_{n \in \N}$ in $(t, b)$, such that $(\xnorm{f(t_{n+1})-f(t_n)})_{n \in \N}$ does not converge to 0. This implies
	\[ \sum_{k=1}^{\infty} \xnorm{f(t_{k+1})-f(t_k)} = \infty , \]
	but since $f$ is of bounded pointwise variation, we get
	\[ \sum_{k=1}^{n} \xnorm{f(t_{k+1})-f(t_k)} \le \Var(f) \]
	for all $n \in \N$.
	This is a contradiction and therefore $f(t+)$ exists for all $t \in [a, b)$. Analogously $f(t-)$ exists for all $t \in (a, b]$.
	
	To prove, that \Cref{bv-a-e-continuous} holds for all $t \in (a,b)$ except on a countable set, we define
	\[ S_n \coloneqq \set{t \in (a, b) \colon \xnorm{f(t-)-f(t)} + \xnorm{f(t)-f(t+)} \ge \frac{1}{n}} \]
	for $n \in \N$.
	For every finite collection of elements $t_1, t_2, \ldots, t_N \in S_n$ with $t_1 < t_2 < \ldots < t_N$, $N \in \N_0$, we can choose $a < \hat t_1 < \hat t_2 \ldots < \hat t_{3N}  < b$, such that for all $k \in \set{1, \ldots, N}$
	\begin{align*}
		\hat t_{3k-1} &= t_k , \\
		\xnorm{f(\hat t_{3k-2}) - f(\hat t_{3k-1})} &\ge \frac{1}{2} \xnorm{f(t_k-) - f(t_k)} \quad \text{and} \\
		\xnorm{f(\hat t_{3k-1}) - f(\hat t_{3k})} &\ge \frac{1}{2} \xnorm{f(t_k) - f(t_k+)} .
	\end{align*}
	Then
	\begin{align*}
		\frac{N}{n} &= \sum_{k=1}^N \frac{1}{n} \\
		&\le \sum_{k=1}^{N} ( \xnorm{f(t_k-)-f(t_k)} + \xnorm{f(t_k)-f(t_k+)} ) \\
		&\le \sum_{k=1}^{3N} 2 ( \xnorm{f(\hat t_{3k-2}) - f(\hat t_{3k-1})} + \xnorm{f(\hat t_{3k-1}) - f(\hat t_{3k})}) \\
		&\le 2 \sum_{k=0}^{3N-1} \xnorm{f(\hat t_{k+1}) - f(\hat t_k)} \\
		&\le 2 \Var(f) .
	\end{align*}
	Therefore $N \le 2 n \Var(f)$, so $S_n$ is finite for all $n \in \N$. Hence
	\[ \set{t \in (a, b) \colon f(t+) \neq f(t) \text{ or } f(t) \neq f(t-)} = \bigcup_{n \in \N} S_n \]
	is a countable union of finite sets and therefore countable.
\end{proof}

\begin{lemma} \label{BV-shift-estimate}
	Let $f \in BV(a, b ; X)$. Then
	\[ \int_a^{b-h} \xnorm{f(\tau+h) - f(\tau)} \diff\tau \le h \essVar (f) \]
	for all $h > 0$.
\end{lemma}

\begin{proof}
	Let $f_r$ be a representative of $f$ and let $h > 0$.
	For $\tau \ge 0$ we set 
	\begin{align*}
		N_\tau &\coloneqq \Floor{\frac{b-a-\tau}{h}} \quad \text{and} \\*
		t_{\tau, k} &\coloneqq \tau + a + k h \text{ for } k \in \set{0, \ldots, N_{\tau}} .
	\end{align*}
	Then
	\begin{align*}
		\int_a^{b-h} \xnorm{f(\tau+h) - f(\tau)} \diff\tau 
		&= \sum_{k=0}^{ N_0 } \int_{a+kh}^{(a+(k+1)h) \land b} \xnorm{f_r(\tau+h) - f_r(\tau)} \diff\tau \\
		&= \sum_{k=0}^{N_0-1} \int_{0}^{h} \xnorm{f_r(t_{\tau, k+1}) - f_r(t_{\tau, k})} \diff\tau \\*
		&\phantom{{}={}} {} + \int_{0}^{b-a-N_0 h} \xnorm{f_r(t_{\tau, N_\tau}) - f_r(t_{\tau, N_\tau-1})} \diff\tau \\
		&= \int_0^h \sum_{k=0}^{N_\tau-1} \xnorm{f_r(t_{\tau, k+1}) - f_r(t_{\tau, k})} \diff\tau \\
		&\le \int_0^h \Var(f_r) \diff\tau \\
		&= \Var(f_r) h .
	\end{align*}
	Taking the infimum over all representatives $f_r$ of $f$ gives the result.
\end{proof}

Note that if $f \in BV(a, b; X)$, then $f$ has a representative $f_r$ that is of bounded pointwise variation. By \Cref{BV-continuity}, for every $t \in [a, b)$ the limit from the right $f_r(t+)$ exists. If $\widehat{f}_r$ is also a representative of $f$ that is of bounded pointwise variation, then $f_r - \widehat{f}_r$ is also of bounded pointwise variation and $f_r - \widehat{f}_r = 0$ almost everywhere. Therefore $(f_r - \widehat{f}_r)(t+) = 0$ and hence
\[ f_r (t+) = (f_r - \widehat{f}_r) (t+) + \widehat{f}_r (t+) = \widehat{f}_r (t+) . \]
We have shown, that $f_r(t+)$ does not depend on the choice of $f_r$. Therefore the following definition is well-defined.

\begin{definition}
	Let $f \in BV(a, b; X)$ and $t \in [a, b)$. Then
	\[ f(t+) \coloneqq f_r(t+) , \]
	where $f_r$ is a representative of $f$ that is of bounded pointwise variation.
\end{definition}

We could replace the $L^1$-norm in the definition of the $BV$-norm (see \Cref{BV-norm-definition}) by $\xnorm{f(a+)}$ and get an equivalent norm. 
More precisely, we obtain the following lemma.

\begin{lemma} \label{BV-norm-equivalence}
	Let $f \in BV(a, b; X)$. Then
	\[ \frac{1}{b-a+1} \norm{f}_{BV(a, b; X)} \le \xnorm{f(a+)} + \essVar(f) \le \klammern{2 \lor \frac{1}{b-a}} \norm{f}_{BV(a, b; X)} . \]
\end{lemma}

\begin{proof}
	Let $f \in BV(a, b; X)$. Then
	\begin{align*}
		\frac{1}{b-a+1} \norm{f}_{BV(a, b; X)} &= \frac{1}{b-a+1} \klammern{ \int_a^b \xnorm{f (\tau)} \diff\tau + \essVar (f) } \\
		&\le \frac{1}{b-a+1} \klammern{ \int_a^b \klammern{ \xnorm{f (\tau) - f(a+)} + \xnorm{f(a+)} } \diff\tau + \essVar (f) } \\
		&\le \frac{1}{b-a+1} \klammern{ (b-a) \klammern{ \essVar(f) + \xnorm{f(a+)} } + \essVar (f) } \\
		&\le \xnorm{f(a+)} + \essVar(f) \\
		&= \frac{1}{b-a} \int_{a}^{b} \xnorm{f(a+)} \diff\tau + \essVar(f) \\
		&\le \frac{1}{b-a} \int_{a}^{b} \klammern{ \xnorm{f(a+) - f(\tau)} + \xnorm{f(\tau)} } \diff\tau + \essVar(f) \\
		&\le \frac{1}{b-a} \klammern{(b-a) \essVar(f) + \norm{f}_{L^1(a, b; X)}} + \essVar(f) \\
		&= \frac{1}{b-a} \norm{f}_{L^1(a, b; X)} + 2 \essVar(f) \\
		&\le \klammern{2 \lor \frac{1}{b-a}} \norm{f}_{BV(a, b; X)} .
	\end{align*}
	This completes the proof.
\end{proof}

The following theorem is a result from Camille Jordan, who introduced the notion of bounded variation in \cite{Jordan81} in 1881.

\begin{theorem}[Jordan decomposition]
	A function $f \colon [a, b] \to \R$ is of bounded pointwise variation if and only if there exist two increasing functions $f_+, f_- \colon [a, b] \to \R$, such that \[ f = f_+ - f_- . \numberthis \]
\end{theorem}

\begin{proof}
	If $f \colon [a, b] \to \R$ is of bounded pointwise variation, then for $t \in [a, b]$ we can define
	\begin{align*}
		f_+ (t) & \coloneqq f(t) + \Var( f |_{ [a, t] }) \quad \text{and} \quad f_- (t) \coloneqq \Var( f |_{ [a, t] } ) .
	\end{align*}
	For all $t, \hat t \in [a, b]$ with $t < \hat t$ we now have
	\begin{align*}
		f_+ (t) &= f(\hat t) + f(t) - f(\hat t) + \Var( f |_{ [a, t] } ) \\
		&\le f(\hat t) + \Var( f |_{ [t, \hat t] } ) + \Var( f |_{ [a, t] } ) \\
		&= f_+ (\hat t) .
	\end{align*}
	Thus, $f_+$ is increasing. Also $f_-$ is increasing and $f = f_+ + f_-$.
	
	To prove the implication in the other direction, note that for any $g, h \colon [a, b] \to \R$ we have
	\[ \Var(g + h) \le \Var(g) + \Var(h) . \]
	If we assume that $f_+, f_- \colon [a, b] \to \R$ are increasing, then 
	\[ \Var(f_+ - f_-) \le \Var(f_+) + \Var(-f_-) = f_+(b) - f_+(a) + f_-(a) - f_-(b) < \infty , \]
	so $f_+ - f_-$ is of bounded pointwise variation.
\end{proof}

\begin{lemma}
    If $f \in C^1([a, b])$, then
    \[ \Var (f) = \essVar(f) = \int_a^b |f'(\tau)| \diff\tau . \numberthis \label{wegintegral} \]
\end{lemma}

\begin{proof}
	Since $f'$ is continuous, the set $\Omega \coloneqq \set{t \in (a, b) \colon f'(t) > 0}$ is open, so there exists a sequence of intervals $(I_n)_{n \in \N}$, such that
	\begin{align*}
		I_{2n} \text{ is open for all }&n \in \N , & \bigcup_{n\in\N} I_{2n} &= \Omega , \\
		I_{2n-1} \text{ is closed for all }&n \in \N \quad \text{and} & \bigcup_{n\in\N} I_{2n-1} &= [a, b] \setminus \Omega .
	\end{align*}
	Let $(a_n)_{n\in\N}$ and $(b_n)_{n\in\N}$ be sequences in $[a, b]$ such that
	\[ I_{2n} = (a_{2n}, b_{2n}) \quad \text{and} \quad I_{2n-1} = [a_{2n-1}, b_{2n-1}] \]
	for all $n \in \N$.
	Now, $f|_{I_n}$ is monotone and $\Var(f|_{I_n}) = \Var((-f)|_{I_n})$ for all $n \in \N$. Therefore
	\begin{align*}
		\int_a^b |f'(\tau)| \diff\tau &= \int_\Omega f'(\tau) \diff\tau - \int_{[a, b] \setminus \Omega} f'(\tau) \diff\tau \\
		&= \sum_{n=1}^\infty \klammern{\int_{I_{2n}} f'(\tau) \diff\tau - \int_{I_{2n-1}} f'(\tau) \diff\tau} \\
		&= \sum_{n=1}^\infty \klammern{ f(b_{2n}) - f(a_{2n}) - (f(b_{2n-1}) - f(a_{2n-1})) } \\
		&= \sum_{n=1}^\infty \klammern{\Var(f|_{I_{2n}}) + \Var(f|_{I_{2n-1}})} \\
		&= \Var(f) .
	\end{align*}
	Since $f$ is continuous, $\Var(f) = \essVar (f)$.
\end{proof}

\paragraph{\bf Acknowledgment.} The author is grateful to Ralph Chill for his valuable insights, thoughtful comments, and constructive suggestions, which have played a significant role in shaping this work.

\bibliographystyle{plain}
\def\cprime{$'$}
  \def\ocirc#1{\ifmmode\setbox0=\hbox{$#1$}\dimen0=\ht0 \advance\dimen0
  by1pt\rlap{\hbox to\wd0{\hss\raise\dimen0
  \hbox{\hskip.2em$\scriptscriptstyle\circ$}\hss}}#1\else {\accent"17 #1}\fi}
  \def\cprime{$'$} \def\cprime{$'$} \def\cprime{$'$}

\end{document}